\documentclass{amsart}

\usepackage{dha}

\author{Jeroen Maes}
\author{Fernando Muro}
\address{Universidad de Sevilla,
Facultad de Matemáticas,
Departamento de Álgebra,
Avda. Reina Mercedes s/n,
41012 Sevilla, Spain}
\email{fmuro@us.es}
\urladdr{http://personal.us.es/fmuro}

\title{Derived homotopy algebras}


\begin{document}

\begin{abstract}
    We develop a theory of minimal models for algebras over an operad defined over a commutative ring, not necessarily a field, extending and supplementing the work of Sagave for the associative case.
\end{abstract}

\maketitle

\tableofcontents

\section{Introduction}

Over a field, minimal models for operadic algebras go back to Kadeishvili's theorem \cite{kadeishvili_1988_structure_infty_algebra}. He showed that the homology graded algebra $H_*(A)$ of a differential graded associative algebra $A$ carries an essentially unique $\Oinf{\O{A}}$-algebra structure, consisting of degree $r-2$ maps, $r\geq 3$,
\[m_r\colon H_*(A)^{\otimes^r}\To H_*(A)\]
satisfying certain equations,
which is $\infty$-quasi-isomorphic to $A$. This structure on $H_*(A)$ is minimal because it has trivial differential. The extension of this result to algebras over Koszul operads in characteristic zero follows from the homotopy transfer theorem \cite[\S10.3]{loday_vallette_2012_algebraic_operads}. This is possible over a field because all modules are projective, hence any chain complex is homotopy equivalent to its homology, but this does not happen over a general commutative ring.

Over a commutative ring $\kk$, Sagave \cite{sagave_2010_dgalgebras_derived_algebras} had the clever idea of replacing complexes with $\kk$-projective bicomplexes whose vertical homology is also $\kk$-projective. It is well known that any chain complex, regarded as a bicomplex concentrated in horizontal degree $0$, has a $\kk$-projective resolution of that kind concentrated in non-negative horizontal degrees \cite[\S XVII.1]{cartan_eilenberg_1956_homological_algebra}.

\begin{center}
    \begin{tikzpicture}
        \draw[->, gray] (-0.2,0) -- (5,0); 
        \draw[->, gray] (0,-2.8) -- (0,3); 
        \foreach \x in {0,...,4}
        \foreach \y in {-2,...,2}
            {
                \node[shape=circle,fill=black,scale=.4] (\x-\y) at (\x,\y) {}; 
            }
        \foreach \x in {1,...,4}
        \foreach \y in {-2,...,2}
            {
                \draw[->, thick] (\x-.1,\y) -- (\x-.9,\y); 
            }
        \foreach \y in {-2,...,2}
            {
                \draw[->, thick] (4.5,\y) -- (4.1,\y); 
            }
        \foreach \x in {0,...,4}
        \foreach \y in {-1,...,2}
            {
                \draw[->, thick] (\x,\y-.1) -- (\x,\y-.9); 
            }
        \foreach \x in {0,...,4}
            {
                \draw[->, thick] (\x,2.5) -- (\x,2.1); 
                \draw[thick] (\x,-2.1) -- (\x,-2.4); 
            }
        \node[rotate=90] at (-0.5,0) {complex};
        \node at (2,2.8) {bicomplex};
        \fill[gray, opacity = .2] (-0.3,-2.4) rectangle (0.3,2.5);
    \end{tikzpicture}
\end{center}
Sagave used simplicial techniques to extend this to associative algebras, and then he constructed a bigraded minimal model $\minimal{A}$ for any differential graded associative algebra $A$. This minimal model is a bicomplex which carries a structure resembling $\Oinf{\O{A}}$-algebras that he called derived $\Oinf{\O{A}}$-algebra. It consists of bidegree $(-i,i-1)$ morphisms, $i\geq 2$,
\[d_i\colon \minimal{A}\To\minimal{A},\]
and bidegree $(i,r-2-i)$ maps, $r\geq 2$, $i\geq 0$,
\[m_{i,r}\colon\minimal{A}^{\otimes^r}\To \minimal{A},\]
satisfying certain equations.
It is minimal in the sense that the vertical differential vanishes. Moreover, the horizontal differential of $\minimal{A}$ is a projective resolution of $H_*(A)$. It was later discovered by Cirici, Egas Santander, Livernet, and Whitehouse that derived $\Oinf{\O{A}}$-algebras are the same as split filtered $\Oinf{\O{A}}$-algebras \cite[Theorem 4.56]{cirici_egas_santander_livernet_whitehouse_2018_derived_infinity_algebras}. The minimal model is equipped with a filtration preserving $\infty$-morphism $\minimal{A}\leadsto A$ which induces an isomorphism between the $E^2$ pages of the source and target bicomplex spectral sequences.

In this paper we extend Sagave's results to algebras $A$ over a Koszul operad $\O{O}$ (in the sense of \cite{fresse_2004_koszul_duality_operads}) with trivial differential such that $\O{O}$ is non-symmetric or $\Q\subset\kk$. We start with the Cartan--Eilenberg model structure on bicomplexes $\bichainCE$ defined in \cite{muro_roitzheim_2019_homotopy_theory_bicomplexes}, whose cofibrant resolutions are $\kk$-projective and have $\kk$-projective vertical homology. In \S\ref{cartan_eilenberg_section}, we transfer this model structure to $\O{O}$-algebras in bicomplexes, as a semi-model structure, so we can take a cofibrant resolution $\cofres{A}\to A$ here. In \S\ref{derived_operads_section}, we define an operad $\D{\O{O}}$ in the category $\grchain$ of graded chain complexes whose algebras are the same as $\O{O}$-algebras in bicomplexes. We show that $\D{\O{O}}$ is Koszul (Theorem \ref{main_main}), so we can apply the homotopy transfer theorem to a contraction of $\cofres{A}$ onto its vertical homology to obtain a minimal model $\minimal{A}$ of $A$, see \S\ref{section:minimal_models}. This minimal model carries a $\Oinf{\D{\O{O}}}$-algebra structure and it is essentially unique. This essential uniqueness, which is established at the end of \S\ref{cobar_section}, is what requires the aforementioned hypotheses on $\O{O}$, for most of the paper we can do with less.

For $\O{O}=\O{A}$ the associative operad, $\D{\O{A}}$ had already been considered by Livernet, Roitzheim, and Whitehouse in \cite{livernet_roitzheim_whitehouse_2013_derived_infty_algebras}. They also computed $\Oinf{\D{\O{A}}}$ and showed that their algebras coincide with Sagave's derived $\Oinf{\O{A}}$-algebras. For other operads, we define derived $\Oinf{\O{O}}$-algebras as $\Oinf{\D{\O{O}}}$-algebras concentrated in non-negative horizontal degrees instead. We then show that they coincide with split filtered $\Oinf{\O{O}}$-algebras, as in the associative case, see \S\ref{presentation}.

We avoid Sagave's simplicial techniques, since the Dold--Kan correspondence is not strictly symmetric \cite{richter_2003_symmetry_properties_doldkan} and this could cause problems with symmetric operads. Our approach via Koszul duality theory has the advantage that it yields closed formulas for the structure maps of a minimal model, unlike the inductive construction used by Kadeishvili and Sagave. Moreover, we show that most of the times we do not need a full cofibrant resolution $\cofres{A}$ to proceed (Remark \ref{minimal_remark}). We illustrate this with Example \ref{example_dugger_shipley}, fully computing a minimal model for the Dugger--Shipley integral differential graded associative algebra \cite{dugger_shipley_2009_curious_example_triangulatedequivalent},
\[A=\frac{\Z\langle e,x^{\pm1}\rangle}{(e^2,ex+xe-x^2)},\qquad
    \abs{e}=\abs{x}=1,\qquad d(e)=p,\quad d(x)=0,\]
where $p\in\Z$ is a prime. This minimal model is remarkably small, although non-trivial,
\begin{align*}
    \minimal{A}             & =\frac{\Z\langle x^{\pm1}, c\rangle}{(c^2,cx+xc)}, &
    (\abs{x}_h,\abs{x}_v)   & =(0,1),                                            &
    (\abs{y}_h,\abs{y}_v)   & =(1,0),                                                                                    \\
                            &                                                    &
    d_h(x)                  & =0,                                                &
    d_h(c)                  & =-p,                                                                                       \\
                            &                                                    & m_{1,2}(cx^{2i-1},x^j) & =x^{2i+j}, &
    m_{1,2}(cx^{2i-1},cx^j) & =cx^{2i+j}.
\end{align*}
The map $m_{0,2}$ is the product of the associative algebra $\minimal{A}$ and the non-indicated operations are trivial.
As far as we know, this is the third full computation of a non-trivial derived $\Oinf{\O{A}}$-algebra in the literature, after \cite[Example 5.1]{sagave_2010_dgalgebras_derived_algebras} and \cite[\S5.2]{aponte_roman_livernet_robertson_whitehouse_ziegenhagen_2015_representations_derived_infinity}. Sagave's example, though, is non-trivial because it has a non-formal underlying complex (see Remark \ref{formal_complex}) unlike Dugger--Shipley's whose non-formality is related to the associative product. We do not know if the example of
Aponte Rom\'an, Livernet, Robertson, Whitehouse, and Ziegenhagen is non-formal.
We compute a commutative and a Lie example too (Examples \ref{commutative_minimal} and \ref{Lie_minimal}).

We also address the strictification process, consisting of reconstructing any differential graded $\O{O}$-algebra $A$ from a minimal model $\minimal{A}$, up to quasi-isomorphism. Sagave did this in the associative case, using a non-functorial construction due to Kontsevich and Soibelman \cite[\S6.2]{kontsevich_soibelman_2009_notes_infty_algebras} since he actually works with strictly unital derived $\Oinf{\O{A}}$-algebras. We instead define a functorial construction $\tot\bar\cobar_{\D{\O{O}}}\bbar_{\D{\O{O}}}\minimal{A}$ based on the bar-cobar adjunction, which is also used to prove the essential uniqueness of minimal models. This $\tot\bar\cobar_{\D{\O{O}}}\bbar_{\D{\O{O}}}\minimal{A}$ is actually a cofibrant replacement of $A$ (Theorem \ref{recover_dg}, Remark \ref{cofibrant_resolution}). As an application, we prove that the set of homotopy classes of differential graded $\O{O}$-algebra maps $A\to B$ is a quotient of the set of derived $\infty$-morphisms $\minimal{A}\leadsto B$ from a minimal model of the source (Corollary \ref{representatives}). This is new even for the associative operad. We apply it to give an elementary proof of the fact that Dugger--Shipley's $A$ is not formal. Further examples will be computed in a forthcoming paper on universal Massey products for operadic algebras over a ring.

Some of the results in \S\ref{derived_operads_section} and \S\ref{koszul}, most notably Theorem \ref{main_main}, were obtained by the first author in his thesis \cite{maes_2016_derived_homotopy_algebras}, directed by the second author. The proofs here are substantially different and simpler. We here use the operadic Koszul complex while \cite{maes_2016_derived_homotopy_algebras} uses the bar construction.

We assume that the reader is familiar with operad theory and Koszul duality theory for operads. We mostly refer to \cite{fresse_2004_koszul_duality_operads} and other papers by Fresse, since we work over a commutative ring. Nevertheless, we borrow terminology and notation from \cite{loday_vallette_2012_algebraic_operads} because we feel it is nowadays more common.

\section{Cartan--Eilenberg resolutions of algebras over an operad}\label{cartan_eilenberg_section}

A Cartan--Eilenberg resolution of a chain complex is a bicomplex which yields projective resolutions of chain, cycle, boundary, and homology modules, see \cite[\S XVII.1]{cartan_eilenberg_1956_homological_algebra}. In this section, we construct Cartan--Eilenberg resolutions with algebraic structure for $\O{O}$-algebras in $\chain$.

\begin{definition}\label{bicomplexes}
    An \emph{unbounded bicomplex} $X$ is a bigraded module equipped with horizontal and vertical differentials,
    \begin{align*}
        d_h\colon X_{p,q} & \To X_{p-1,q}, &
        d_v\colon X_{p,q} & \To X_{p,q-1},
    \end{align*}
    satisfying $d_hd_v+d_vd_h=0$. The \emph{horizontal} and \emph{vertical degrees} of an element $x\in X_{p,q}$ are $\abs{x}_h=p$ and $\abs{x}_v=q$, respectively, and the \emph{total degree} is $\abs{x}=p+q$. The \emph{bidegree} of $x$ is the pair $\abs{x}_b=(p,q)$. A morphism of unbounded bicomplexes is a bigraded morphism $f\colon X\to Y$ which commutes with horizontal and vertical differentials. We denote the category of unbounded bicomplexes by $\bichainu$.
    This category is equipped with a closed symmetric monoidal structure given by
    \[(X\otimes Y)_{p,q}=\bigoplus_{\substack{i+s=p\\j+t=q}}X_{i,j}\otimes Y_{s,t}.\]
    and
    \begin{align*}
        d_h(x\otimes y) & =d_h(x)\otimes y+(-1)^{\abs{x}}x\otimes d_h(y)  \\
        d_v(x\otimes y) & =d_v(x)\otimes y+(-1)^{\abs{x}}x\otimes d_v(y).
    \end{align*}
    Moreover, the symmetry constraint uses the Koszul sign rule with respect to the total degree,
    \[x\otimes y\mapsto (-1)^{\abs{x}\abs{y}}y\otimes x.\]

    A \emph{bounded bicomplex}, or just \emph{bicomplex} in this paper, is an unbounded bicomplex $X$ which vanishes in negative horizontal degrees, $X_{p,q}=0$, $p<0$. The full subcategory formed by these objects will be denoted by $\bichain$. This category inherits the closed monoidal structure from $\bichainu$.
\end{definition}

It may look strange to keep the name \emph{bicomplex} for bounded ones but these are the only ones we need in the homotopical part. Unbounded complexes are just used from a combinatorial viewpoint.

\begin{remark}\label{vertical_inclusion}
    A chain complex can be regarded as a bicomplex concentrated in horizontal degree $0$. This yields a full inclusion
    \begin{equation*}
        \chain\hookrightarrow\bichain
    \end{equation*}
    which strictly preserves tensor products and related constraints, like $\bichain\hookrightarrow\bichainu$. We can use it to consider algebras in $\bichain$ over an operad $\O{O}$ in $\chain$.
    The category of $\O{O}$-algebras in $\bichain$ is denoted by $\algebras{\bichain}{\O{O}}$. The previous full inclusion extends to $\O{O}$-algebras,
    \[\algebras{\chain}{\O{O}}\hookrightarrow\algebras{\bichain}{\O{O}}.\]
\end{remark}

Cartan--Eilenberg resolutions of $\O{O}$-algebras in $\chain$ will live in $\algebras{\bichain}{\O{O}}$. We use some homotopical techniques to construct them.

Given an adjoint pair $F\dashv U$,
\begin{equation*}
    \mathscr{M}\mathop{\rightleftarrows}_U^F\mathscr{C}
\end{equation*}
with $\mathscr{M}$ a model category, we say that $\mathscr{C}$ admits the \emph{transferred model structure} if it has a (necessarily unique) model structure where a morphism $f$ in $\mathscr{C}$ is a weak equivalence or (trivial) fibration if and only if the morphism $U(f)$ is so in $\mathscr{M}$.

\emph{Semi-model categories} are one of the many relaxations of the axioms of a model category, albeit a very close one, see \cite[\S12.1.3]{fresse_2009_modules_operads_functors}. Given an adjunction as above, where $\mathscr{M}$ is a cofibrantly generated honest model category with given sets of generating (trivial) cofibrations, we say that $\mathscr{C}$ admits the \emph{transferred semi-model structure} if the hypotheses of \cite[Theorem 12.1.4]{fresse_2009_modules_operads_functors} are satisfied.

We consider the adjoint pair
\begin{equation}\label{underlying_bicomplex}
    \bichain\mathop{\rightleftarrows}\limits^{\O{O}\circ-}_U\algebras{\bichain}{\O{O}}.
\end{equation}
Here, $U$ is the obvious forgetful functor and $\O{O}\circ-$ denotes the free $\O{O}$-algebra functor. We endow $\bichain$ with the Cartan--Eilenberg combinatorial model structure in \cite[\S4]{muro_roitzheim_2019_homotopy_theory_bicomplexes}, denoted by $\bichainCE$, and with the sets of generating (trivial) cofibrations therein. Here, cofibrant resolutions are Cartan--Eilenberg resolutions and weak equivalences are $E^2$-equivalences, i.e.~bicomplex morphisms which induce an isomorphism in the $E^2$ term of the spectral sequence associated to the filtration by the horizontal degree. We recall the definition of $E^2$-equivalences in a wider context in Definition \ref{minimal_models} below.

A \emph{graded operad} is an operad in $\chain$ with trivial differential.

\begin{proposition}\label{semi-model}
    Let $\O{O}$ be a graded operad such that each $\O{O}(r)$ is $\SS_r$-projective.
    The category $\algebras{\bichainCE}{\O{O}}$ admits the transferred semi-model structure from the Cartan--Eilenberg model structure $\bichainCE$ along \eqref{underlying_bicomplex}. Moreover, the forgetful functor $U\colon\algebras{\bichainCE}{\O{O}}\to\bichainCE$ preserves cofibrant objects and cofibrations with cofibrant source.
\end{proposition}

\begin{proof}
    This is just an application of \cite[Theorem 12.3.A and Proposition 12.3.2]{fresse_2009_modules_operads_functors}. We just have to check that $\O{O}$ is $\SS$-cofibrant as an operad in $\bichainCE$, i.e.~its underlying $\SS$-module is cofibrant in $\bichainCE^{\SS}$ with the projective model structure.

    Let $S^{0,q}$ be the bicomplex which is just $\kk$ concentrated in bidegree $(0,q)$, $q\in\Z$. The maps $0\rightarrowtail S^{0,q}$ are some of the generating cofibrations of $\bichainCE$. Hence, any $\SS$-module $M$ in $\bichainCE$ concentrated in horizontal degree $0$ with trivial differential and such that each $M(r)$ is $\SS_r$-projective, $r\geq0$, is cofibrant in $\bichainCE^{\SS}$.
\end{proof}

\begin{remark}\label{is_cartan_eilenberg}
    Cofibrant objects and cofibrant resolutions in semi-model categories behave as in model categories. Given an $\O{O}$-algebra $A$ in $\chain$, a cofibrant resolution of $A$ in $\algebras{\bichainCE}{\O{O}}$ is a Cartan--Eilenberg resolution by Proposition \ref{semi-model} and \cite[\S4]{muro_roitzheim_2019_homotopy_theory_bicomplexes}.
\end{remark}

In some cases, the semi-model structure in Proposition \ref{semi-model} is an honest model category structure, e.g.~when $\O{O}=\O{A}$ is the associative operad, whose algebras are non-unital monoids, or $\O{O}=u\O{A}$ is the unital associative operad, whose algebras are honest (unital) monoids.

\begin{proposition}
    The categories $\algebras{\bichainCE}{\O{A}}$ and $\algebras{\bichainCE}{u\O{A}}$ admit the transferred model structure along \eqref{underlying_bicomplex} from $\bichainCE$.
\end{proposition}

This follows from \cite[Proposition 4.2]{muro_roitzheim_2019_homotopy_theory_bicomplexes} and \cite[Theorem 1.2]{muro_2011_homotopy_theory_nonsymmetric}, see also \cite{muro_2017_correction_articles_homotopy}.

We should mention that related model structures on twisted complexes have been considered in \cite{livernet_whitehouse_ziegenhagen_2020_spectral_sequence_associated,fu_guan_livernet_whitehouse_2020_model_category_structures,cirici_santander_livernet_whitehouse_undefined/ed_model_category_structures}. Weak equivalences are defined in terms of the pages of the spectral sequence of the twisted complex. None of these model structures coincides with any our (semi-)model structures here or in \cite{muro_roitzheim_2019_homotopy_theory_bicomplexes}.

\section{Algebras in bicomplexes and their operads}\label{derived_operads_section}

Let $\O{O}$ be a graded operad. Here we describe $\O{O}$-algebras in $\bichainu$ as operadic algebras in the following category of graded complexes. For the associative operad, this was done in \cite{livernet_roitzheim_whitehouse_2013_derived_infty_algebras}. This will allow us to define minimal models from the Cartan--Eilenberg resolutions given by Proposition \ref{semi-model}.

\begin{definition}
    The category $\grchain$ of \emph{graded complexes} is the product category $\chain^\Z$. An object $X$ is a sequence of complexes
    \[X=(X_{p,*})_{p\in\Z}=(\dots, X_{p,*},X_{p+1,*},\dots ).\]
    The \emph{horizontal}, \emph{vertical}, and \emph{total} degrees, and the \emph{bidegree}, are defined as in Definition \ref{bicomplexes}. The differential $d$ of $X$ preserves the horizontal degree and decreases the vertical degree by one,
    \[d\colon X_{p,q}\To X_{p,q-1}.\]
    A graded complex $X$ is \emph{bounded} if it is concentrated in non-negative horizontal degrees, i.e.~$X_{p,*}=0$ if $p<0$.


    The category $\grchain$ is equipped with a closed symmetric monoidal structure with tensor product defined also as in Definition \ref{bicomplexes}.
\end{definition}

\begin{remark}
    The full inclusion in horizontal degree $0$,
    \[\chain\hookrightarrow\grchain,\]
    strictly preserves tensor products and related constraints.
    An unbounded bicomplex $X$ has an underlying graded complex $(X,d_v)$.
\end{remark}


\begin{definition}
    The quadratic algebra of \emph{dual numbers} $\O{D}=\kk[\epsilon]/(\epsilon^2)$ is the monoid in $\grchain$ generated by $\epsilon$ in bidegree $(-1,0)$ with relation and $\epsilon^2=0$ and differential $d(\epsilon)=0$. We regard it as an operad concentrated in arity $1$.
\end{definition}

\begin{remark}\label{dual_numbers}
    A $\O{D}$-algebra $X$ in $\grchain$ is an unbounded bicomplex. The vertical differential $d_v$ is the graded complex differential of $X$ and $d_h(x)=\epsilon(x)$, $x\in X$.
\end{remark}

Recall the general notion of \emph{operadic distributive law}, see e.g.~\cite[\S8.6.1]{loday_vallette_2012_algebraic_operads}. They are used to construct a new operad from two existing ones. This concept goes back to \cite{beck_1969_distributive_laws}.

\begin{definition}\label{distributive}
    Let $\O{O}$ be a graded operad.
    The operadic distributive law
    \[\varphi\colon \O{D}\circ\O{O}\To\O{O}\circ\O{D}\]
    is the morphism of $\SS$-modules defined as follows for any $\mu\in\O{O}(r)$,
    \begin{align*}
        \varphi(1;\mu)        & =(\mu;1,\dots,1),                                                                                  \\
        \varphi(\epsilon;\mu) & =(-1)^{\abs{\mu}}\sum_{i=1}^{r}(\mu;1,\stackrel{i-1}{\dots},1,\epsilon,1,\stackrel{r-i}{\dots},1).
    \end{align*}

    A \emph{bigraded operad} is an operad ind $\grchain$ with trivial differential. The \emph{derived operad} $\D{\O{O}}$ is the bigraded operad $\O{O}\circ_{\varphi}\O{D}$ with underlying $\SS$-module $\O{O}\circ\O{D}$ and composition
    \begin{center}
        \begin{tikzcd}[column sep = 12mm]
            \gamma_{\D{\O{O}}}\colon\O{O}\circ\O{D}\circ \O{O}\circ\O{D}\ar[r,"\id{}\circ\varphi\circ\id{}"]&\O{O}\circ\O{O}\circ \O{D}\circ\O{D}
            \ar[r,"\gamma_{\O{O}}\circ\gamma_{\O{D}}"]&
            \O{O}\circ\O{D}.
        \end{tikzcd}
    \end{center}
    The operadic unit is $(1;1)$.
\end{definition}

\begin{lemma}\label{distributive_is}
    The morphism $\varphi$ is indeed an operadic distributive law.
\end{lemma}

\begin{proof}
    According to \cite[\S8.6.1]{loday_vallette_2012_algebraic_operads} we have to check that four diagrams, called (i), (ii), (I), and (II), commute. The commutativity of (i) is equivalent to the first formula in Definition \ref{distributive}. For (ii) we use that $\varphi(\epsilon;1)=(1;\epsilon)$ by the second formula.

    In order to check (I) we take an element of $\O{D}\circ\O{O}\circ\O{O}$, that we can denote in either way
    \[
        (\epsilon;(\mu;\nu_{1},\dots, \nu_{r}))
        =
        ((\epsilon;\mu);\nu_{1},\dots, \nu_{r})
    \]
    according to how we associate the triple composite. Here $\nu_i\in\O{O}(s_i)$.
    On the one hand, if we apply the operad composition $\gamma_{\O{O}}\colon \O{O}\circ\O{O}\to \O{O}$ to the part in $\O{O}\circ\O{O}$ and then $\varphi$ we obtain
    \begin{align*}
        \varphi(\epsilon;\mu(\nu_{1},\dots, \nu_{r}) ) & =                                                                                               (-1)^{\abs{\mu}+\sum_{i=1}^{r}\abs{\nu_i}}\sum_{i=1}^{\sum_{j=1}^{r}s_j}\mu(\nu_{1},\dots, \nu_{r})(1,\stackrel{i-1}{\dots},1,\epsilon,1,\dots,1).
    \end{align*}
    On the other hand,
    \begin{multline*}
        (\varphi(\epsilon;\mu);\nu_{1},\dots, \nu_{r})=
        (-1)^{\abs{\mu}}\sum_{i=1}^{r}((\mu;1,\stackrel{i-1}{\dots},1,\epsilon,1,\stackrel{r-i}{\dots},1);\nu_{1},\dots, \nu_{r})\\
        =
        \sum_{i=1}^{r}(-1)^{\abs{\mu}+\sum_{j=1}^{i-1}\abs{\nu_j}}(\mu;(1;\nu_1),\dots,(1;\nu_{i-1}),(\epsilon;\nu_i),(1;\nu_{i+1}),\dots,(1;\nu_{r})).
    \end{multline*}
    If we now apply $\varphi$ to all factors in $\O{D}\circ\O{O}$ we obtain
    \begin{multline*}
        \sum_{i=1}^{r}\sum_{j=1}^{s_i}(-1)^{\abs{\mu}+\sum_{j=1}^{i}\abs{\nu_j}}(\mu;(\nu_1;1,\dots,1),\dots,(\nu_i;1,\stackrel{j-1}{\dots},\epsilon,\dots,1),\dots,(\nu_{r};1,\dots,1))\\
        =(-1)^{\abs{\mu}+\sum_{i=1}^{r}\abs{\nu_i}}\sum_{i=1}^{\sum_{j=1}^{r}s_j}((\mu;\nu_{1},\dots, \nu_{r});(1,\stackrel{i-1}{\dots},1,\epsilon,1,\dots,1)).
    \end{multline*}
    If we now apply $\gamma_{\O{O}}$ we obtain the same result as above.

    Finally, for (II) we consider
    \[(\epsilon;(\epsilon;\mu))=((\epsilon;\epsilon);\mu).\]
    I we apply the operad composition $\gamma_{\O{D}}$ of $\O{D}$ we obtain $0$ because $\epsilon^2=0$. Moreover,
    \begin{align*}
        (\epsilon;\varphi(\epsilon;\mu)) & =(-1)^{\abs{\mu}}\sum_{i=1}^{r}(\epsilon;(\mu;1,\stackrel{i-1}{\dots},\epsilon,\dots,1))  \\
                                         & =(-1)^{\abs{\mu}}\sum_{i=1}^{r}((\epsilon;\mu);1,\stackrel{i-1}{\dots},\epsilon,\dots,1),
    \end{align*}
    and
    \begin{multline*}
        (-1)^{\abs{\mu}}\sum_{i=1}^{r}(\varphi(\epsilon;\mu);1,\stackrel{i-1}{\dots},\epsilon,\dots,1)
        =\sum_{i,j=1}^{r}((\mu;1,\stackrel{j-1}{\dots},\epsilon,\dots,1);1,\stackrel{i-1}{\dots},\epsilon,\dots,1)\\
        =\sum_{1\leq j<i\leq n} (\mu;(1;1),\dots,\underbrace{(\epsilon;1)}_{j \text{ slot}},\dots,\underbrace{(1;\epsilon)}_{i \text{ slot}},\dots,(1;1))\\
        +(\mu;(1;1),\dots,(\epsilon;\epsilon),\dots,(1;1))\\
        -\sum_{1\leq i<j\leq n} (\mu;(1;1),\dots,\underbrace{(1;\epsilon)}_{i \text{ slot}},\dots,\underbrace{(\epsilon;1)}_{j \text{ slot}},\dots,(1;1)).
    \end{multline*}
    If we now apply $\gamma_{\O{D}}$ and use that $\epsilon^2=0$ the result is again
    \[\sum_{1\leq j<i\leq n} (\mu; 1,\dots, \overset{j}{\epsilon},\dots, \overset{i}{\epsilon},\dots, 1)
        -\sum_{1\leq i<j\leq n} (\mu; 1,\dots, \overset{i}{\epsilon},\dots, \overset{j}{\epsilon},\dots, 1)=0.\]
\end{proof}


\begin{proposition}\label{dO-algebra}
    A $\D{\O{O}}$-algebra in $\grchain$ is the same as an $\O{O}$-algebra in the category of unbounded bicomplexes.
\end{proposition}

\begin{proof}
    Let $\unit$ be the initial operad in $\grchain$, which is the monoidal unit for the circle product $\circ$ in $\grchain^{\SS}$, and let $\eta_{\O{P}}\colon\unit\to\O{P}$ be the unit of a generic operad $\O{P}$ in $\grchain$.

    We have operad morphisms induced by the units of $\O{O}$ and $\O{D}$, respectively,
    \begin{center}
        \begin{tikzcd}[column sep = 14mm]
            \O{O}\cong
            \O{O}\circ\unit
            \ar[r,"\id{\O{O}}\circ\eta_{\O{D}}"]&
            \O{O}\circ_{\varphi}\O{D}&
            \unit\circ\O{D}\cong\O{D}\ar[l,"\eta_{\O{O}}\circ\id{\O{O}}"'].
        \end{tikzcd}
    \end{center}
    Therefore, any $\D{\O{O}}$-algebra $A$ is an $\O{O}$-algebra in $\grchain$ as well as a bicomplex, see Remark \ref{dual_numbers}.

    The second equation of Definition \ref{distributive} implies
    \[d_h(\mu(x_1,\dots,x_r))=\sum_{s=1}^r(-1)^{\abs{\mu}+\sum_{t=1}^{s-1}\abs{x_t}}\mu(x_1,\dots, d_h(x_s),\dots,x_r),\]
    This shows that $A$ is not only an $\O{O}$-algebra as a graded complex but as an unbounded bicomplex. The converse is essentially the same.
\end{proof}

Recall that the \emph{quadratic operad} $\P{E}{R}$ generated by a \textit{reduced} $\SS$-module $E$, i.e.~$E(0)=0$, is the quotient $\F(E)/(R)$ of the free operad $\F(E)$ generated by $E$ by a given sub-$\SS$-module $R\subset E\circ_{(1)}E=\F(E)^{(2)}$ of weight $2$ relations, see \cite[\S5.2.5]{fresse_2004_koszul_duality_operads} and \cite[\S7.1.2]{loday_vallette_2012_algebraic_operads}.

\begin{proposition}\label{derived_presentation}
    The derived operad of a reduced quadratic operad $\O{O}=\P{E}{R}$ is quadratic and has the following presentation
    \[\D{\O{O}}=\P{E\oplus\kk\cdot\epsilon}{R\oplus D\oplus \kk\cdot\epsilon^2}.\]
    Here $D\subset  \kk\cdot\epsilon\circ_{(1)}E\oplus E\circ_{(1)}\kk\cdot\epsilon$ is the sub-$\SS$-module defined as
    \[D(r)=\Big\{\epsilon\circ_1\mu-(-1)^{\abs{\mu}}\sum_{i=1}^r\mu\circ_i\epsilon;\quad \mu\in E(r)\Big\}.\]
\end{proposition}

\begin{proof}
    Let $\O{P}$ be the operad defined by the presentation in the statement.
    On the one hand, there is an operad morphism $\O{P}\to\D{\O{O}}$ defined on generators by the inclusions
    \[E\cong E\circ\unit\To \O{O}\circ_{\varphi}\O{D}\longleftarrow\unit\circ\kk\cdot\epsilon\cong\kk\cdot\epsilon.\]
    These arrows are given by the inclusions of generators $E\subset\O{O}$ and $\kk\cdot\epsilon\subset\O{D}$ and the corresponding operadic units.

    On the other hand, it is easy to check that
    \[(\O{O}\circ\O{D})(r)=\O{O}(r)\otimes\O{D}^{\otimes^r}\]
    as $\SS_r$-modules, with the permutation action on the tensor power and the diagonal action on the tensor product. Hence, there is an $\SS$-module morphism
    \begin{align*}
        \O{O}\circ\O{D}                           & \To\O{P},   \\
        (\mu;\epsilon^{i_1},\dots,\epsilon^{i_r}) & \;\mapsto\;
        \mu(\epsilon^{i_1},\dots,\epsilon^{i_r}).
    \end{align*}
    Here, each $i_j\in\{0,1\}$. This morphism is surjective because of the relations in $D$.

    The composite
    \[\O{O}\circ\O{D}\To\O{P}\To\D{\O{O}}\]
    is clearly the identity on the underlying $\SS$-module $\O{O}\circ\O{D}$. Since the first one is surjective, both are isomorphisms.
\end{proof}

\section{Koszul duality}\label{koszul}


We here prove that, if $\O{O}$ is a Koszul graded operad, then its derived operad $\D{\O{O}}$ from Definition \ref{distributive} is also Koszul. Moreover, we compute its Koszul dual $\K{\D{\O{O}}}$ in terms of the Koszul dual $\K{\O{O}}$ of $\O{O}$. We will use it in the next section to compute the structure of a minimal model of a differential graded $\O{O}$-algebra arising from a Cartan--Eilenberg resolution.

We always consider cooperads with respect to the usual composite $\circ$ of $\SS$-modules since all of them are reduced in this paper, see \cite[\S5.1.15]{loday_vallette_2012_algebraic_operads} and \cite[Proposition 1.1.15]{fresse_2000_homotopy_simplicial_algebras}.

Given a graded operad $\O{O}$ and a graded cooperad $\O{C}$, a \emph{twisting morphism} $\tm\colon\O{C}\to\O{O}$, called twisting cochain in \cite[\S4.5]{fresse_2004_koszul_duality_operads}, is a degree $-1$ morphism of $\SS$-modules such that the composite
\begin{center}
    \begin{tikzcd}
        \tm\star\tm\colon \O{C}\ar[r,"\Delta_{(1)}"]&
        \O{C}\circ_{(1)}\O{C}\ar[r,"\tm\circ_{(1)}\tm"]&
        \O{O}\circ_{(1)}\O{O}\ar[r,"\gamma_{(1)}"]&
        \O{O}
    \end{tikzcd}
\end{center}
vanishes, see \cite[\S6.4]{loday_vallette_2012_algebraic_operads}. Here $\Delta_{(1)}$ is the infinitesimal decomposition of the cooperad $\O{C}$ and $\gamma_{(1)}$ is the infinitesimal composition of the operad $\O{O}$, see \cite[\S6.1.4, \S6.1.2]{loday_vallette_2012_algebraic_operads}.

We will denote the decomposition $\Delta$ of a cooperad $\O{C}$ by using a Sweedler notation like in \cite[\S5.8.1]{loday_vallette_2012_algebraic_operads},
\begin{equation}\label{sweedler_decomposition}
    \Delta\colon\K{\O{O}}\To \K{\O{O}}\circ\K{\O{O}},\qquad
    \Delta(\mu)=\sum_{[\mu]}(\nu;\nu^{1},\dots,\nu^{l})\cdot\tau.
\end{equation}

The \emph{(left) twisted composite product} $\O{C}\circ_\tm\O{O}$ \cite[\S6.4.5]{loday_vallette_2012_algebraic_operads} is the $\SS$-module $\O{C}\circ\O{O}$ equipped with the following differential
\begin{multline}\label{twisted_complex_differential}
    d(\rho;\mu_1,\dots,\mu_r)=\sum_{i=1}^r \sum_{[\mu]}(-1)^{\abs{\rho}+\abs{\nu_i}\sum_{j=1}^{i-1}\abs{\mu_j}}(\rho\circ_i\tm(\nu_i);\dots,\mu_{i-1},\nu_i^1,\dots\\\dots,\nu_i^l,\mu_{i+1},\dots)\cdot (\id{}\times \tau\times \id{}).
\end{multline}
Here the two $\id{}$'s are identity permutations, the first one on the sum of the arities of the $\mu_k$ for $k<i$, and similarly the second one for $k>i$. In terms of labeled trees, this differential can be better described as
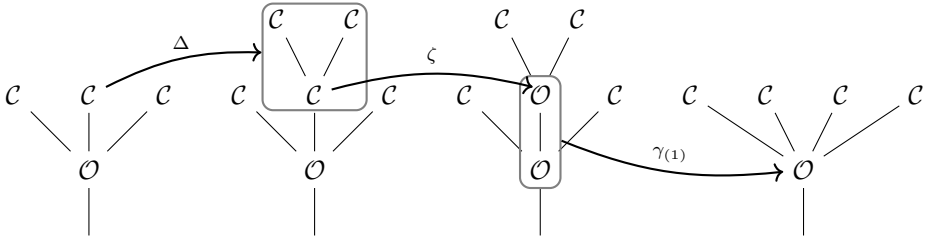
\begin{center}
    \begin{tikzpicture}[grow'=up,sibling distance=10mm, level distance=10mm]
        \node [draw=none] {} child{ node {$\mathcal{O}$}
                child{ node {$\mathcal{C}$} }
                child{ node (A) {$\mathcal{C}$} }
                child{ node {$\mathcal{C}$} }
            };
        \node at (3,0) [draw=none] {} child{ node {$\mathcal{O}$}
                child{ node {$\mathcal{C}$} }
                child{ node (a) {$\mathcal{C}$}
                        child{ node (b) {$\mathcal{C}$} }
                        child{ node (c) {$\mathcal{C}$} }
                    }
                child{ node {$\mathcal{C}$} }
            };
        \node at (6,0) [draw=none] {} child{ node (d) {$\mathcal{O}$}
                child{ node {$\mathcal{C}$} }
                child{ node (e) {$\mathcal{O}$}
                        child{ node{$\mathcal{C}$} }
                        child{ node {$\mathcal{C}$} }
                    }
                child{ node {$\mathcal{C}$} }
            };
        \node at (9.5,0) [draw=none] {} child{ node (D) {$\mathcal{O}$}
                child{ node {$\mathcal{C}$} }
                child{ node{$\mathcal{C}$} }
                child{ node {$\mathcal{C}$} }
                child{ node {$\mathcal{C}$} }
            };
        \node [draw=gray, thick, rounded corners, inner sep=-1pt, fit=(a) (b) (c)] (B) {};
        \node [draw=gray, thick, rounded corners, inner sep=0pt, fit=(d) (e)] (C) {};
        \draw[->, draw=gray, thick] (A) edge [bend left=5mm] node [above] {$\scriptstyle \Delta$} (B);
        \draw[->, draw=gray, thick] (a) edge [bend left=5mm] node [above] {$\scriptstyle \tm$} ($(e)+(-.1,.1)$);
        \draw[->, draw=gray, thick] (C) edge [bend right=5mm] node [above] {$\scriptstyle \gamma_{(1)}$} (D);
    \end{tikzpicture}
\end{center}
Beware that the formula for the differential of $\O{O}\circ_{\zeta}\O{C}$ in \cite[\S4.5.1]{fresse_2004_koszul_duality_operads} is not fully correct as written (it does not even have degree $-1$). Here, we have used the correct definition from \cite[\S6.4.5]{loday_vallette_2012_algebraic_operads}. We also borrow the tree notation notation from \cite[6.4.5]{loday_vallette_2012_algebraic_operads}, where it is used to describe the differential of the right twisted composite product $\O{C}\circ_{\zeta}\O{O}$ which we do not use here.

Suppose both $\O{O}$ and $\O{C}$ are weighted, \emph{reduced}, and \emph{connected}, i.e.~their weight $0$ part is $\unit$, e.g.~quadratic (co)operads generated by a reduced $\SS$-module. Then we ask $\tm$ to be weight-preserving and we say that it is \emph{Koszul} if the twisted composite product $\O{O}\circ_\tm\O{C}$ is acyclic in positive weight. It is nevertheless customary to just say that it is \emph{acyclic} since in weight zero it is always $\unit\circ\unit\cong\unit$, which is obviously not acyclic.

The \emph{suspension} $\suspension X$ of a chain complex is defined by shifting one degree up and changing the sign of the differential. Equivalently, $\suspension$ is defined as a functor by the existence of a natural degree $1$ isomorphism $\suspension\colon X\to \suspension X$.

The \emph{Koszul dual cooperad} $\K{\O{O}}=\C{sE}{s^2R}$ of a reduced quadratic operad $\P{E}{R}$ is cogenerated by $sE$ with corelations $s^2R$. The \emph{canonical twisting morphism} \[\ctm\colon\K{\O{O}}\to\O{O}\] is the composite
\[\K{\O{O}}\twoheadrightarrow sE\stackrel{s^{-1}}{\To}E\hookrightarrow\O{O}\]
given by the the projection onto the cogenerators (the weight $1$ part), the desuspension, and the inclusion of the generators (again the weight $1$ part). The operad $\O{O}$ is \emph{Koszul} if both $\O{O}$ and $\K{\O{O}}$ are $\kk$-projective and $\ctm$ is Koszul, i.e.~the \emph{operadic Koszul complex}, which is the the twisted composite product $\O{O}\circ_{\ctm}\K{\O{O}}$, is acyclic. We will use the following characterization instead.

\begin{lemma}\label{characterization_koszul}
    A reduced quadratic operad $\O{O}$ in $\grmodules$ is Koszul if and only if there exists a weighted, reduced, and connected cooperad $\O{C}$ in $\grmodules$ and a weight preserving twisting morphism $\tm\colon\O{C}\to\O{O}$ such that both $\O{O}$ and $\O{C}$ are $\kk$-projective, $\tm$ is acyclic, and $\tm$ vanishes in weight $\neq 1$. In this case, the Koszul dual cooperad of $\O{O}$ is $\K{\O{O}}=\O{C}$.
\end{lemma}

\begin{proof}
    The ``only if" part follows by definition. If we now assume the existence of $\O{C}$ with the required properties then there is an adjoint cooperad morphism $\bar{\tm}\colon\O{C}\to B\O{O}$ \cite[Theorem 6.5.7]{loday_vallette_2012_algebraic_operads}. Here $B\O{O}$ is the bar construction, see \cite[\S3.1.9]{fresse_2004_koszul_duality_operads} and \cite[\S6.5.1]{loday_vallette_2012_algebraic_operads}. The morphism $\bar{\tm}$ is a quasi-isomorphism by the left module version of \cite[Theorem 2.1.13]{loday_vallette_2012_algebraic_operads} and \cite[Lemma 4.7.2]{loday_vallette_2012_algebraic_operads}, see also \cite[Theorem 6.6.1]{loday_vallette_2012_algebraic_operads}. Since $\O{C}$ has trivial differential (it is a cooperad in $\grmodules$), it is the homology of the bar construction $B\O{O}$. Moreover, since $\tm$ vanishes in weights $\neq 1$, $\bar{\tm}$ maps to the maximal weight part of $B\O{O}$ on each bar degree, see \cite[Lemma 5.2.2]{fresse_2004_koszul_duality_operads}. Therefore, the lemma follows, see \cite[\S5.2.3]{fresse_2004_koszul_duality_operads}.
\end{proof}

We can regard a graded complex $X$ as a complex equipped with a splitting
\[X_n=\bigoplus_{p+q=n}X_{p,q}\]
such that $d(X_{p,q})\subset X_{p,q-1}$. In this way, it is easy to see that, if we start with an object in $\grchain$,
most constructions in $\chain$ performed in Koszul duality theory remain in $\grchain$. Therefore, we can freely use notions and results from Koszul duality theory in \cite{fresse_2004_koszul_duality_operads,fresse_2009_operadic_cobar_constructions,loday_vallette_2012_algebraic_operads} in this graded context.

The suspension of a graded complex, regarded as a sequence of complexes, is defined component-wise. This means shifting the vertical degree and changing the sign of the differential. The suspension construction extends to $\SS$-modules in the obvious way.

We would like to use Lemma \ref{characterization_koszul} to show that the derived operad $\D{\O{O}}$ of a reduced quadratic operad $\O{O}=\P{E}{R}$ is Koszul, and compute its Koszul dual $\K{\D{\O{O}}}$. In order to find a candidate for this cooperad, we define the distributive law in Definition \ref{codistributive} below.

\begin{proposition}\label{dual_numbers_koszul}
    The algebra $\O{D}$, regarded as an operad concentrated in arity $1$, is Koszul. Its Koszul dual $\K{\O{D}}=\kk[\delta]$ is the polynomial coalgebra on one generator $\delta$ of bidegree $(-1,1)$ with trivial differential.
\end{proposition}

\begin{proof}
    Let $\ctm\colon\kk[\delta]\to\kk[\epsilon]/(\epsilon^2)$ be the morphism of vertical degree $-1$ defined by $\ctm(\delta)=\epsilon$ and $\ctm(\delta^r)=0$ for $r\neq 1$. This is a twisting morphism. The differential of the corresponding Koszul complex $\kk[\epsilon]/(\epsilon^2)\otimes_{\ctm}\kk[\delta]$ is $d(1\otimes\delta^r)=\epsilon\otimes\delta^{r-1}$ for $r>0$, $d(1\otimes 1)=0$, and $d(\epsilon\otimes\delta^r)=0$ for $r\geq 0$. This is clearly acyclic.
\end{proof}





A \emph{cooperadic distributive law} is the dual categorical notion of operadic distributive law. They can be used to construct a new cooperad from two given ones. A \textit{(bi)graded cooperad} is a cooperad in $\operatorname{(Gr)Ch}$ with trivial differential.

\begin{definition}\label{codistributive}
    Let $\O{C}$ be a graded cooperad.
    The cooperadic distributive law
    \[\K{\varphi}\colon \K{\O{D}}\circ\O{C}
        \To \O{C}\circ\K{\O{D}}\]
    is the morphism of $\SS$-modules in $\grchain$ defined as follows for any $\mu\in\O{C}(r)$ and $i\geq 0$,
    \[\K{\varphi}(\delta^i;\mu)=\sum_{j_1+\cdots+j_r=i}(\mu;\delta^{j_1},\dots,\delta^{j_r}).\]

    The \emph{derived cooperad} $\D{\O{O}}$ of $\O{C}$ is the bigraded cooperad $\K{\O{D}}\circ_{\K{\varphi}}\O{C}$, with underlying $\SS$-module $\K{\O{D}}\circ\O{C}$ and decomposition
    \begin{center}
        \begin{tikzcd}[column sep=14mm]
            \K{\O{D}}\circ\O{C}\ar[r,"\Delta_{\K{\O{D}}}\circ\Delta_{\O{C}}"]& \K{\O{D}}\circ\K{\O{D}}\circ\O{C}\circ\O{C}
            \ar[r,"\id{}\circ\K{\varphi}\circ\id{}"]&\K{\O{D}}\circ\O{C}\circ \K{\O{D}}\circ\O{C}.
        \end{tikzcd}
    \end{center}
    If $\varepsilon_{\O{C}}\colon\O{C}\to\unit$ denotes the counit of $\O{C}$, the counit of $\K{\O{D}}\circ_{\K{\varphi}}\O{C}$ is
    $\varepsilon_{\K{\O{D}}}\circ\varepsilon_{\O{C}}$. 
\end{definition}

\begin{lemma}\label{is_codistributive}
    The morphism of $\SS$-modules $\K{\varphi}$ in Definition \ref{codistributive} is indeed a cooperadic distributive law.
\end{lemma}

\begin{proof}
    We have to check that the categorical duals of the four diagrams in  \cite[\S8.6.1]{loday_vallette_2012_algebraic_operads} commute. These diagrams are called (i), (ii), (I), and (II). By the very definition,
    \begin{align*}
        \K{\varphi}(1;\mu)      & =(\mu;1,\dots,1), &
        \K{\varphi}(\delta^i;1) & =(1;\delta^i).
    \end{align*}
    We choose an element $(\delta^i;\mu)\in(\K{\O{D}}\circ\K{\O{O}})(r)$ and decompose the $\K{\O{O}}$ term, using the Sweedler notation in Remark \ref{infinity_morphisms},
    \begin{align*}
        (\delta^i,\Delta(\mu)) & =\sum_{[\mu]}(\delta^i;(\nu;\nu^{1},\dots,\nu^{l})\cdot\tau)  \\
                               & =\sum_{[\mu]}((\delta^i;\nu);\nu^{1},\dots,\nu^{l})\cdot\tau.
    \end{align*}
    We apply $\K{\varphi}$ to the $\K{\O{D}}\circ\K{\O{O}}$ component
    \begin{align*}
        \sum_{[\mu]}(\K{\varphi}(\delta^i;\nu);\nu^{1},\dots,\nu^{l})\cdot\tau & =
        \sum_{[\mu]}\sum_{j_1+\cdots+j_l=i}((\nu;\delta^{j_1},\dots,\delta^{j_l});\nu^{1},\dots,\nu^{l})\cdot\tau                                                                    \\
                                                                               & =\sum_{[\mu]}\sum_{j_1+\cdots+j_l=i}(\nu;(\delta^{j_1};\nu^1),\dots,(\delta^{j_l};\nu^l))\cdot\tau.
    \end{align*}
    We now apply $\K{\varphi}$ to all factors in $\K{\O{D}}\circ\K{\O{O}}$. If $\nu^j\in\K{\O{O}}(s_j)$, this yields
    \begin{multline*}
        \hspace{-8pt}\sum_{[\mu]}\sum_{j_1+\cdots+j_l=i}\big(\nu;\hspace{-10pt}\sum_{k_{1,1}+\dots+k_{1,s_1}=j_1}\hspace{-16pt}(\nu^1;\delta^{k_{1,1}},\dots,\delta^{k_{1,s_1}}),\dots,\hspace{-16pt}\sum_{k_{l,1}+\dots+k_{l,s_l}=j_l}\hspace{-16pt}(\nu^l;\delta^{k_{l,1}},\dots,\delta^{k_{l,s_l}})\big)\cdot\tau.
    \end{multline*}

    The composite $\O{C}\circ\K{\O{D}}$ is given by
    \[(\O{C}\circ\K{\O{D}})(r)=\O{C}(r)\otimes(\K{\O{D}})^{\otimes^r}\]
    as $\SS_r$-modules, with the permutation action on the tensor power and the diagonal action on the tensor product, like in the proof of Proposition \ref{derived_presentation}.

    Since necessarily $\sum_{j=1}^ls_j=r$, the previous summation coincides with
    \begin{multline*}
        \sum_{[\mu]}\sum_{j_1+\cdots+j_r=i}((\nu;\nu^{1},\dots,\nu^{l});\delta^{j_1},\dots,\delta^{j_r})\cdot\tau\\
        =\sum_{j_1+\cdots+j_r=i}\sum_{[\mu]}((\nu;\nu^{1},\dots,\nu^{l})\cdot\tau;\delta^{j_{\tau(1)}},\dots,\delta^{j_{\tau(r)}}).
    \end{multline*}
    Changing variables, this equals
    \[\sum_{j_1+\cdots+j_r=i}\sum_{[\mu]}((\nu;\nu^{1},\dots,\nu^{l})\cdot\tau;\delta^{j_{1}},\dots,\delta^{j_{r}}).\]
    This is the same as if we first consider
    $\K{\varphi}(\delta^i,\mu)$ in Definition \ref{codistributive} and then decompose $\mu\in\K{\O{O}}(r)$. This proves the dual of (II).

    We can also start by applying the decomposition to the $\K{\O{D}}$ term $\delta^i$,
    \begin{align*}
        (\Delta(\delta^i),\mu) & =\sum_{j+k=i}((\delta^j;\delta^k);\mu)  \\
                               & =\sum_{j+k=i}(\delta^j;(\delta^k;\mu)),
    \end{align*}
    then we apply $\varphi$,
    \begin{align*}
        \sum_{j+k=i}(\delta^j;\varphi(\delta^k;\mu)) & =
        \sum_{j+k=i}\sum_{l_1+\dots+l_r=k}(\delta^j;(\mu;\delta^{l_1},\dots,\delta^{l_r})) \\
                                                     & =
        \sum_{j+k=i}\sum_{l_1+\dots+l_r=k}((\delta^j;\mu);\delta^{l_1},\dots,\delta^{l_r})),
    \end{align*}
    and $\varphi$ again,
    \begin{multline*}
        \sum_{j+k=i}\sum_{l_1+\dots+l_r=k}(\varphi(\delta^j;\mu);\delta^{l_1},\dots,\delta^{l_r})\\
        =\sum_{j+k=i}\sum_{l_1+\dots+l_r=k}\sum_{m_1+\dots+m_r=j}((\mu;\delta^{m_1},\dots,\delta^{m_r});\delta^{l_1},\dots,\delta^{l_r})\\
        =\sum_{j+k=i}\sum_{l_1+\dots+l_r=k}\sum_{m_1+\dots+m_r=j}(\mu;(\delta^{m_1};\delta^{l_1}),\dots,(\delta^{m_r};\delta^{l_r}))\\
        =\sum_{l_1+\dots+l_r+m_1+\dots+m_r=i}(\mu;(\delta^{m_1};\delta^{l_1}),\dots,(\delta^{m_r};\delta^{l_r})).
    \end{multline*}
    This is the same as if we take $\varphi(\delta^i;\mu)$ in Definition \ref{codistributive} and decompose all terms in $\K{\O{D}}$,
    \begin{multline*}
        \sum_{j_1+\cdots+j_r=i}\big(\mu;\sum_{m_1+l_1=j_1}(\delta^{m_1};\delta^{l_1}),\dots,\sum_{m_r+l_r=j_r}(\delta^{m_r};\delta^{l_r})\big)\\
        =\sum_{m_1+l_1+\cdots+m_r+l_r=i}(\mu;(\delta^{m_1};\delta^{l_1}),\dots,(\delta^{m_r};\delta^{l_r})).
    \end{multline*}
    This checks the dual of (I).

    In order to show the commutativity of the dual of (i), we take $\varphi(\delta^i;\mu)$ in Definition \ref{codistributive}. Then we apply the counit $\varepsilon_{\K{\O{D}}}\colon\K{\O{D}}\to\unit$ to the $\K{\O{D}}$ factors, which yields $(\mu;1,\dots,1)$ if $i=0$ and $0$ otherwise. We finally apply the unit isomorphism $\K{\O{O}}\circ\unit\cong\K{\O{O}}$, which yields $\mu$ in the first case. The result is the same as if we first apply the counit $\varepsilon_{\K{\O{D}}}$ to the $\K{\O{D}}$ factor of $(\delta^i;\mu)$, which yields $(1;\mu)$ if $i=0$ and $0$ otherwise, and then the unit isomorphism $\unit\circ\K{\O{O}}\cong\K{\O{O}}$.

    The commutativity of the dual of (ii) is easy because the only $\K{\O{O}}$ term appearing in $\varphi(\delta^i;\mu)$ is $\mu$ itself, see Definition \ref{codistributive}. Therefore, on the one hand, if $\mu$ has positive weight, we obtain zero either if we apply $\varepsilon_{\K{\O{O}}}$ to all $\mu$'s in $\varphi(\delta^i;\mu)$ or if we directly apply it to the only $\mu$ in $(\delta^i;\mu)$. On the other hand, if $\mu=1$ is the operadic unit, then $\varepsilon_{\K{\O{O}}}(1)=1$, $\varphi(\delta^i;1)=(1,\delta^i)$, and both $(\delta^i;1)$ and $(1,\delta^i)$ map to $\delta^i$ via the unit isomorphisms $\K{\O{D}}\circ\unit\cong\K{\O{D}}\cong\unit\circ\K{\O{D}}$.
\end{proof}

We are almost ready to prove the main theorem of this section. Before, we need to recall a spectral sequence comparison result which will be used several times.

\begin{definition}\label{filtered_complex}
    A \emph{filtered complex} is a chain complex $X$ equipped with an increasing filtration $F_*X$ by subcomplexes indexed by the integers,
    \[\cdots\subset F_nX\subset F_{n+1}X\subset\cdots\subset X.\]
    This filtration is \emph{exhaustive} if $X=\bigcup_nF_nX$ and \emph{bounded below} if $F_{N}X=0$ for certain $N$. A \emph{morphism} $f\colon X\to Y$ of filtered complexes is a chain map such that $f(F_nX)\subset F_nY$.
\end{definition}

\begin{remark}\label{spectral_sequence_remark}
    Bounded below filtered complexes are \emph{complete} in the sense of \cite[Definitions 3.1 and 3.8]{mccleary_2001_user_guide_spectral}.

    A filtered complex $X$ has an associated spectral sequence with $E^0$-term
    \[E^0_{p,q}(X)=\frac{F_pX_{p+q}}{F_{p-1}X_{p+q}},\]
    see \cite[\S2.2]{mccleary_2001_user_guide_spectral}.
    A morphism of filtered complexes 
    induces (functorially) a map between the corresponding spectral sequences.

    We will also consider decreasingly filtered cochain complexes in some proofs below, and their associated cohomological spectral sequences.
\end{remark}

\begin{lemma}[{\cite[Theorem 3.9]{mccleary_2001_user_guide_spectral}}]\label{spectral}
    If $f\colon X\to Y$ is a morphism of complete and exhaustive filtered complexes which induces an isomorphism on some page $E^n$ of the corresponding spectral sequences, then $f$ is a quasi-isomorphism.
\end{lemma}

\begin{theorem}\label{main_main}
    If $\O{O}$ is graded Koszul and arity-wise projective over the corresponding symmetric group, then $\D{\O{O}}$ is Koszul and its Koszul dual is the derived cooperad of the Koszul dual of $\O{O}$, $\K{(\D{\O{O}})}=\D{(\K{\O{O}})}$.
\end{theorem}

\begin{proof}
    Let $\tm\colon\K{\O{D}}\circ_{\K{\varphi}}\K{\O{O}}\to\O{O}\circ_{\varphi}\O{D}=\D{\O{O}}$ be the composite
    \[\K{\O{D}}\circ_{\K{\theta}}\K{\O{O}}\twoheadrightarrow
        \unit\circ sE\oplus\kk\cdot\delta\circ\unit\cong sE\oplus\kk\cdot\delta\stackrel{s^{-1}}{\To}
        E\oplus\kk\cdot\epsilon\cong
        E\circ\unit\oplus\unit\circ\kk\cdot\epsilon\hookrightarrow
        \O{O}\circ_{\theta}\O{D}.\]
    These arrows are the projection onto the weight $1$ part, the desuspension, and the inclusion of the weight $1$ part, respectively, i.e.~$\zeta$ vanishes in weight $\neq 1$, and in weight $1$ it is given by
    \begin{align*}
        \tm(\delta;1)          & = (1;\epsilon),                   &
        \tm(1;\suspension \mu) & = (\mu;1,\dots,1),\quad \mu\in E.
    \end{align*}
    This degree $-1$ morphism is a twisting morphism, i.e.~the composite $\zeta\star\zeta$
    \begin{center}
        \begin{tikzcd}[column sep=7mm]
            \K{\O{D}}\circ_{\K{\varphi}}\K{\O{O}}\ar[r,"\Delta_{(1)}"]&
            (\K{\O{D}}\circ_{\K{\varphi}}\K{\O{O}})\circ_{(1)}(\K{\O{D}}\circ_{\K{\varphi}}\K{\O{O}})\ar[r,"\tm\circ_{(1)}\tm"]&
            \O{O}\circ_{\varphi}\O{D}\circ_{(1)}\O{O}\circ_{\varphi}\O{D}\ar[r,"\gamma_{(1)}"]&
            \O{O}\circ_{\varphi}\O{D}
        \end{tikzcd}
    \end{center}
    vanishes. Let us check this claim.

    The infinitesimal decomposition $\Delta_{(1)}$ of a cooperad $\O{C}$ will be denoted by
    \begin{equation}\label{sweedler_infinitesimal_decomposition}
        \Delta_{(1)}\colon\O{C}\To \O{C}\circ_{(1)}\O{C},\qquad
        \nonumber \Delta_{(1)}(\mu)=\sum_{(\mu)}(\mu^{(1)}\circ_{l}\mu^{(2)})\cdot\sigma,
    \end{equation}
    like in \cite[\S10.1.2]{loday_vallette_2012_algebraic_operads}. See also \eqref{sweedler_decomposition}.

    With this notation, the infinitesimal decomposition of $\K{\O{D}}\circ_{\K{\varphi}}\K{\O{O}}$ is
    \[\Delta_{(1)}(\delta^i;\mu)=\sum_{j+k=i}\sum_{(\mu)}((\delta^j;\mu^{(1)})\circ_l(\delta^k;\mu^{(2)}))\cdot\sigma,\qquad\mu\in\K{\O{O}}.\]
    Therefore, $\tm\star\tm$ vanishes except possibly in weight $2$. In weight $2$ it also vanishes by the following formulas,
    \begin{align*}
        \tm\star\tm(\delta^2;1) & = \tm(\delta;1)\circ_1 \tm(\delta;1)=(1;\epsilon)\circ_1(1;\epsilon)=(1;\epsilon^2)=0.
    \end{align*}
    Given $\mu\in E$,
    \begin{align*}
        \tm\star\tm(\delta;s\mu) & = \tm(\delta;1)\circ_{1}\tm(1;s\mu)-(-1)^{\abs{\mu}}\sum_{i=1}^r \tm(1;s\mu)\circ_{i}\tm(\delta;1)                        \\
                                 & = (1;\epsilon)\circ_{1}(\mu;1,\dots,1)-(-1)^{\abs{\mu}}\sum_{i=1}^r (\mu;1,\dots,1)\circ_{i}(1;\epsilon)                  \\
                                 & = \varphi(\epsilon;\mu)-(-1)^{\abs{\mu}}\sum_{i=1}^r (\mu;1,\stackrel{i-1}{\dots},1,\epsilon,1,\stackrel{r-i}{\dots},1)=0
    \end{align*}
    by definition of $\varphi$.
    We have that $(\K{\O{O}})^{(2)}=s^2R$, and
    given $\sum \mu\circ_l\nu\in R\subset E\circ_{(1)}E$,
    \begin{align*}
        \tm\star\tm\left(1;\sum (-1)^{\abs{\mu}}s\mu\circ_ls\nu\right) & = -\sum \tm(1;s\mu)\circ_l\tm(1;s\nu)          \\
                                                                       & = -\sum (\mu;1,\dots,1)\circ_l(\nu;1,\dots,1)  \\
                                                                       & = \left(-\sum \mu\circ_l\nu;1,\dots,1\right)=0
    \end{align*}
    by the relations $R$ in $\O{O}$.

    We now show that the twisted composite product $(\O{O}\circ_{\varphi}\O{D})\circ_{\tm}(\K{\O{D}}\circ_{\K{\varphi}}\K{\O{O}})$ is acyclic in positive weights. The decomposition of $\K{\O{D}}\circ_{\K{\varphi}}\K{\O{O}}$, which plays a role in the definition of the previous twisted complex, is
    \[\Delta(\delta^i;\mu)=\sum_{[\mu]}\sum_{j+k_1+\cdots+k_l=i}((\delta^{j};\nu);(\delta^{k_1};\nu^{1}),\dots,(\delta^{k_l};\nu^{l}))\cdot\tau.\]
    Here we use the Sweedler formula for the decomposition of $\K{\O{O}}$ in \eqref{sweedler_decomposition}.
    We increasingly filter $(\O{O}\circ_{\varphi}\O{D})\circ_{\tm}(\K{\O{D}}\circ_{\K{\varphi}}\K{\O{O}})$ according to the weight of $\K{\O{O}}$. This is indeed a filtration because, according to the previous formulas
    \begin{multline*}
        (\zeta\circ\id{})\Delta(\delta^i;\mu)=((1;\epsilon);(\delta^{i-1};\mu))
        \\+\sum_{\substack{[\mu]\\\nu\in sE}}\sum_{k_1+\cdots+k_l=i}(s^{-1}\nu;1,\dots,1);(\delta^{k_1};\nu^{1}),\dots,(\delta^{k_l};\nu^{l}))\cdot\tau,
    \end{multline*}
    and the sum of the weights of the $\nu^i$ is one less than the weight of $\mu$. This filtration is bounded below and exhaustive. Moreover, the previous formula also shows that $E^0((\O{O}\circ_{\varphi}\O{D})\circ_{\tm}(\K{\O{D}}\circ_{\K{\varphi}}\K{\O{O}}))=\O{O}\circ(\O{D}\circ_{\ctm}\K{\O{D}})\circ\K{\O{O}}$, with $d_0$ given just by the Koszul complex differential of $\O{D}$.
    We have that
    \[\O{O}\circ\O{D}\circ\K{\O{D}}\circ\K{\O{O}}
        =
        \bigoplus_{r\geq 0}
        \O{O}(r)\otimes_{\SS_r}(\O{D}\circ\K{\O{D}}\circ\K{\O{O}})^{\otimes^r}.\]
    The homology of $\O{D}\circ_{\ctm}\K{\O{D}}$ is $\unit$ by Proposition \ref{dual_numbers_koszul}. Therefore, the homology of $\O{D}\circ\K{\O{D}}\circ\K{\O{O}}$ is $\K{\O{O}}$. Since each $\O{O}(r)$ is projective as an $\SS_r$-module, we conclude that the homology of $\O{O}\circ(\O{D}\circ_{\ctm}\K{\O{D}})\circ\K{\O{O}}$ is $\O{O}\circ\K{\O{O}}$. Moreover, as a graded complex $E^1((\O{O}\circ_{\varphi}\O{D})\circ_{\tm}(\K{\O{D}}\circ_{\K{\varphi}}\K{\O{O}}))=\O{O}\circ_{\ctm}\K{\O{O}}$ is the Koszul complex of $\O{O}$. Indeed, the spectral sequence differential $d_1$ is the Koszul complex differential because $\O{O}$ is a sub-operad of $\D{\O{O}}=\O{O}\circ_{\varphi}\O{D}$, $\K{\O{O}}$ is a sub-cooperad of $\K{\O{D}}\circ_{\K{\varphi}}\K{\O{O}}$, and the twisting morphism $\zeta$ restricts to the canonical twisting morphism $\kappa\colon \K{\O{O}}\to\O{O}$ on $\K{\O{O}}$. Since $\O{O}$ is Koszul, $E^1$ is acyclic in positive weights, hence so is $(\O{O}\circ_{\varphi}\O{D})\circ_{\tm}(\K{\O{D}}\circ_{\K{\varphi}}\K{\O{O}})$, see Lemma \ref{spectral}. Now this theorem follows from Lemma \ref{characterization_koszul}.
\end{proof}

Henceforth, the Koszul dual of $\D{\O{O}}$ will simply be denoted by $\K{\D{\O{O}}}$.

\begin{remark}\label{loday_vallette_gap}
    In \cite{livernet_roitzheim_whitehouse_2013_derived_infty_algebras}, the authors prove the previous theorem for $\O{O}=\O{A}$ the associative operad using the standard quadratic presentations of $\O{A}$ and $\O{D}$ and \cite[Theorem 8.6.5 and Proposition 8.6.6]{loday_vallette_2012_algebraic_operads}. These results are based on \cite[Theorem 8.6.4]{loday_vallette_2012_algebraic_operads}, whose proof contains a gap. Loday and Vallette use a filtration of the bar construction by number of inversions. This filtration is not compatible with the differential, as we now show with a simple example.

    We consider two copies of the ring of dual numbers, $\kk[x]/(x^2)$ and $\kk[y]/(y^2)$. We regard them as operads in $\chain$ concentrated in arity $1$ and degree $0$. Let \[\lambda\colon \kk\cdot y\otimes\kk\cdot x\To \kk\cdot x\otimes\kk\cdot y\] be the homomorphism (rewriting rule) defined by $\lambda(y\otimes x)=x\otimes y$. With the notation in \cite{loday_vallette_2012_algebraic_operads},
    \[\kk[x]/(x^2)\vee_{\lambda}\kk[y]/(y^2)=\kk[x,y]/(x^2,y^2).\]
    A $\kk$-linear basis of this commutative algebra is $\{1,x,y,xy\}$. The bar construction $\bbar(\kk[x,y]/(x^2,y^2))$ is the (non-commutative) polynomial coalgebra generated by $\{s(x),s(y),s(xy)\}$ in degree $1$. Here, the number of inversions of a monomial consists of counting the occurrences of
    \[\begin{array}{cccc}
            \cdots s(y)s(x)\cdots,  &
            \cdots s(xy)s(x)\cdots, &
            \cdots s(y)s(xy)\cdots, &
            \cdots s(xy)s(xy)\cdots.
        \end{array}\]
    The filtration of the bar construction is by number of inversions. The monomial $s(y)s(y)s(x)s(x)$ has one inversion but its differential $s(y)s(xy)s(x)$ has two.
\end{remark}

\section{Derived homotopy algebras and infinity-morphisms}\label{presentation}

Recall that, given an operad $\O{O}$, its infinity operad $\Oinf{\O{O}}=\cobar\K{\O{O}}$ is obtained by applying the cobar construction to its Koszul dual cooperad.

\begin{definition}
    Let $\O{O}$ be a graded operad. A \emph{derived homotopy $\O{O}$-algebra} or \emph{derived $\Oinf{\O{O}}$-algebra} is a $\Oinf{\D{\O{O}}}$-algebra with underlying bounded graded complex.
\end{definition}

Here and elsewhere $\Oinf{\D{\O{O}}}$ is to be understood as the infinity operad of the derived operad $\D{\O{O}}$, not the derived operad of the infinity operad $\Oinf{\O{O}}$, i.e.~$\Oinf{\D{\O{O}}}=\Oinf{(\D{\O{O}})}\neq\D{(\Oinf{\O{O}})}$. The latter is related but considerably smaller. Hence, if $\O{O}$ is graded Koszul $\Oinf{\D{\O{O}}}$ is the cobar construction of the derived cooperad $\D{\K{\O{O}}}$, see Theorem \ref{main_main}.

We would like to characterize derived homotopy $\Oinf{\D{\O{O}}}$-algebras as split filtered $\Oinf{\O{O}}$-algebras, which is a more familiar structure. For this, we need a down-to-earth description of $\Oinf{\D{\O{O}}}$-algebras in terms of operations and equations.

\begin{remark}\label{alternative}
    For $\O{O}$ a graded operad, an $\Oinf{\O{O}}$-algebra can be described in terms of the infinitesimal decomposition $\Delta_{(1)}$ of the coaugmented cooperad $\K{\O{O}}$, see \cite[\S10.1.2]{loday_vallette_2012_algebraic_operads}. We use the Sweedler notation for $\Delta_{(1)}$ in \eqref{sweedler_infinitesimal_decomposition}.

    An $\Oinf{\O{O}}$-algebra is a complex $A$ equipped with structure morphisms
    \begin{equation*}\label{new_structure_maps}
        \begin{split}
            \K{\O{O}}(r)_{n_0}\otimes A_{n_1}\otimes\cdots\otimes A_{n_r} & \To A_{n_0-1+n_1+\cdots+n_r},  \\
            \mu\otimes x_1\otimes\cdots\otimes x_r                                  & \;\mapsto\;\mu(x_1,\dots,x_r),
        \end{split}
    \end{equation*}
    satisfying
    \begin{equation}\label{alfa}
        (\mu\cdot\sigma)(x_1,\dots,x_r)=(-1)^{\alpha_\sigma}\mu(x_{\sigma^{-1}(1)},\dots,x_{\sigma^{-1}(r)}),\qquad \alpha_\sigma =\sum_{\substack{s<t\\\sigma(s)>\sigma(t)}}\abs{x_s}\abs{x_t},
    \end{equation}
    for any permutation $\sigma\in\SS_r$,
    \begin{multline}\label{equation_infinity}
        d(\mu(x_1,\dots,x_r))+\sum_{s=1}^r(-1)^{\beta}\mu(x_1,\dots,d(x_s),\dots,x_r)
        \\
        +\sum_{(\mu)}(-1)^{\gamma}\mu^{(1)}(x_{\sigma^{-1}(1)},\dots,\mu^{(2)}(x_{\sigma^{-1}(l)},\dots),\dots)=0,
    \end{multline}
    where
    \begin{align}\label{beta_gamma}
        \beta  & =\abs{\mu}+\sum_{t=1}^{s-1}\abs{x_t},                                                       &
        \gamma & =\alpha_\sigma+\abs{\mu^{(1)}}+(\abs{\mu^{(2)}}-1)\sum_{m=1}^{l-1}\abs{x_{\sigma^{-1}(m)}},
    \end{align}
    and
    \[1(x)=0,\]
    where $1\in\K{\O{O}}(1)_0$ is given by the coaugmentation.

    Moreover, $\O{O}$-algebras are the same as $\Oinf{\O{O}}$-algebras with $\mu(x_1,\dots,x_r)=0$ for $\mu\in\K{\O{O}}$ of weight $\geq 2$. This reflects the definition of the canonical quasi-isomorphism $\Oinf{\O{O}}\twoheadrightarrow\O{O}$.

    This description also works in the category of graded complexes. In that case, the structure morphisms have bidegree $(0,-1)$ and the signs are computed by using the total degree.
\end{remark}

If the graded operad $\O{O}$ is Koszul, any $\Oinf{\D{\O{O}}}$-algebra has an underlying $\Oinf{\O{D}}$-algebra since $\K{\O{D}}\subset\D{\K{\O{O}}}$, see Theorem \ref{main_main}. We will start by describing $\Oinf{\O{D}}$-algebras. This was done in \cite[10.3.7]{loday_vallette_2012_algebraic_operads} for the non-bigraded version of $\O{D}$. The bigraded version is very similar.

\begin{definition}\label{twisted}
    A \emph{twisted complex} $X$ is a bigraded module equipped with module morphisms
    \[d_i\colon X_{p,q}\To X_{p-i,q+i-1},\quad i\geq 0,\]
    satisfying
    \begin{equation}\label{twisted_complex_equation}
        \sum_{j+k=i}d_jd_k=0,\quad i\geq 0.
    \end{equation}
\end{definition}

\begin{proposition}\label{d_infinity}
    A $\Oinf{\O{D}}$-algebra is the same thing as a twisted complex.
\end{proposition}

\begin{proof}
    By Proposition \ref{dual_numbers_koszul}, $\K{\O{D}}=\kk[\delta]$ is the polynomial coalgebra on one generator of bidegree $(-1,1)$ with trivial differential. Its infinitesimal decomposition is
    \[\Delta_{(1)}(\delta^i)=\sum_{j+k=i}\delta^j\circ_1\delta^k.\]
    Applying the graded complex version of Remark \ref{alternative}, the correspondence between twisted complexes and $\Oinf{\O{D}}$-algebras is given by the equations
    \begin{align*}
        d_0(x) & =d(x),                  &
        d_i(x) & =\delta^i(x),\quad i>0.
    \end{align*}
    Equation \eqref{twisted_complex_equation} corresponds to \eqref{equation_infinity} for $\K{\O{D}}$ with the previous infinitesimal decomposition.
\end{proof}

\begin{remark}
    Proposition \ref{d_infinity} extends to an equivalence of categories if we define a morphism of twisted complexes $f\colon X\to Y$ as a family of module morphisms $f_{p,q}\colon X_{p,q}\to Y_{p,q}$, $p,q\in\Z$, commuting with all the $d_i$.
\end{remark}

\begin{proposition}\label{dha_char}
    Let $\O{O}$ be a graded Koszul operad.
    A $\Oinf{\D{\O{O}}}$-algebra $A$ is the same as a twisted complex equipped with module morphisms, $i\geq 0$,
    \begin{align*}
        \K{\O{O}}(r)_s\otimes A_{p_1,q_1}\otimes\cdots\otimes A_{p_r,q_r} & \To A_{p_1+\cdots+p_r-i,s-1+q_1+\cdots+q_r+i}, \\
        \mu\otimes x_1\otimes\cdots\otimes x_r                            & \;\mapsto\; \mu_i(x_1,\dots,x_r),
    \end{align*}
    satisfying the following equations for $i\geq 0$:
    \begin{equation}\label{symmetry_equation}
        (\mu\cdot\sigma)_i(x_1,\dots,x_r)=(-1)^{\alpha_\sigma}\mu_i(x_{\sigma^{-1}(1)},\dots,x_{\sigma^{-1}(r)}),\quad \sigma\in\SS_r,
    \end{equation}
    \begin{multline}\label{big_equation}
        \sum_{j+k=i}d_j(\mu_k(x_1,\dots,x_r))+\sum_{s=1}^r(-1)^\beta\sum_{j+k=i}\mu_j(x_1,\dots,d_k(x_s),\dots,x_r)\\
        +\sum_{(\mu)}(-1)^\gamma\sum_{j+k=i}\mu^{(1)}_j(x_{\sigma^{-1}(1)},\dots,\mu^{(2)}_k(x_{\sigma^{-1}(l)},\dots),\dots)=0,
    \end{multline}
    and
    \begin{equation}\label{small_equation}
        1_i(x)=0.
    \end{equation}
    Here, $\alpha,\beta,\gamma$ are as in \eqref{alfa} and \eqref{beta_gamma}.
\end{proposition}

\begin{proof}
    We can split the Sweedler notation for the infinitesimal decomposition of $\O{O}$ in \eqref{sweedler_infinitesimal_decomposition} as follows.
    If $\mu\in\K{\O{O}}$ has positive weight,
    \[\Delta_{(1)}(\mu)=1\circ_i\mu+\sum_{(\mu)'}(\mu^{(1)}\circ_{l}\mu^{(2)})\cdot\sigma+\sum_{s=1}^r\mu\circ_s1,\]
    where $(\mu)'$ stands for the summands where both $\mu^{(1)}$ and $\mu^{(2)}$ have positive weight, and
    \[\Delta_{(1)}(1)=1\circ_11.\]
    The infinitesimal de composition of $\K{\D{\O{O}}}$ is then
    \begin{align*}
        \Delta_{(1)}(\delta^i;\mu)  ={} & \sum_{j+k=i}(\delta^j;1)\circ_1(\delta^k;\mu)
        +\sum_{s=1}^r\sum_{j+k=i}(\delta^j;\mu)\circ_s(\delta^k;1)                                                                 \\
                                        & + \sum_{(\mu)'}\sum_{j+k=i}((\delta^j;\mu^{(1)})\circ_l(\delta^k;\mu^{(2)}))\cdot\sigma,
    \end{align*}
    see Theorem \ref{main_main}.

    By the graded version of Remark \ref{alternative}, a $\Oinf{\D{\O{O}}}$-algebra is a graded complex $A$ with differential $d_0$ equipped with module morphisms
    \begin{align*}
        (\K{\O{D}}\circ_{\K{\varphi}}\K{\O{O}})(r)_{p_0,q_0}\otimes A_{p_1,q_1}\otimes\cdots\otimes A_{p_r,q_r} & \To A_{p_0+p_1+\cdots+p_r,q_0-1+q_1+\cdots+q_r},
    \end{align*}
    satisfying certain equations. Note that $\K{\O{D}}\circ\K{\O{O}}=k[\delta]\otimes\K{\O{O}}$. If we denote
    \[(\delta^i\otimes\mu)(x_1,\cdots, x_r)=\mu_i(x_1,\dots,x_r),\]
    the equations are
    \[(\mu_i\cdot\sigma)(x_1,\dots,x_r)=(-1)^{\alpha_\sigma}\mu_i(x_{\sigma^{-1}(1)},\dots,x_{\sigma^{-1}(r)})\]
    for any permutation $\sigma\in\SS_r$, which coincides with \eqref{symmetry_equation},
    \begin{multline}\label{big_equation_2}
        d_0(\mu_i(x_1,\dots,x_r))+\sum_{t=1}^r(-1)^{\beta}\mu_i(x_1,\dots,d_0(x_t),\dots,x_r)
        \\
        +\sum_{\substack{j+k=i\\j>0}}1_j\mu_k(x_1,\dots, x_r)
        +\sum_{s=1}^r(-1)^\beta\sum_{\substack{j+k=i\\k>0}}\mu_j(x_1,\dots,1_k(x_s),\dots,x_r)\\
        +\sum_{(\mu)'}(-1)^\gamma\sum_{j+k=i}\mu_j^{(1)}(x_{\sigma^{-1}(1)},\dots,\mu_k^{(2)}(x_{\sigma^{-1}(l)},\dots),\dots)=0,
    \end{multline}
    if $\mu$ has positive weight,
    \begin{equation}\label{twisted_complex_equation_2}
        d_0(1_i(x_1))+1_i(d_0(x_1))+\sum_{\substack{j+k=i\\j,k>0}}1_j(1_k(x_1))=0,
    \end{equation}
    and
    \begin{equation}\label{small_equation_2}
        1_0(x)=0.
    \end{equation}
    If we make the following change of variables,
    \begin{align*}
        d_i(x) & = 1_i(x),\quad i>0; \\
        1_i(x) & = 0,\quad i>0;
    \end{align*}
    then \eqref{big_equation_2} turns into \eqref{big_equation}, \eqref{twisted_complex_equation_2} turns into \eqref{twisted_complex_equation}, and \eqref{small_equation_2} turns into \eqref{small_equation}.
\end{proof}

\begin{remark}
    For $\O{O}$ graded Koszul, any $\Oinf{\D{\O{O}}}$-algebra has an underlying $\O{O}$-algebra in $\grchain$ since $\K{\O{O}}\subset\D{\K{\O{O}}}$ by Theorem \ref{main_main}, so $\Oinf{\O{O}}\subset\Oinf{\D{\O{O}}}$. The graded complex differential of this underlying structure is $d_0$ and the structure operations are the $\mu_0$ for $\mu\in\K{\O{O}}$.
\end{remark}

\begin{example}\label{examples_derived_homotopy_algebras}
    We apply the previous proposition to the classical operads $\O{A}$, $\O{C}$, and $\O{L}$, whose algebras are associative, commutative, and Lie algebras, respectively, as well as to the initial operad $\unit$, which is $\kk$ concentrated in arity and degree $0$ (generated by the operadic unit).
    \begin{enumerate}
        \item A \emph{derived homotopy associative algebra} $A$ is a twisted complex equipped with bidegree $(-i,r-2+i)$ operations $m_{i,r}\colon A^{\otimes^r}\to A$, $r\geq 2$, $i\geq 0$,
              satisfying
              \begin{multline*}
                  \sum_{j+k=i}d_j(m_{k,r}(x_1,\dots,x_r))+\sum_{s=1}^r(-1)^{r-1+\sum_{t=1}^{s-1}\abs{x_t}}\sum_{j+k=i}m_{j,r}(x_1,\dots,d_k(x_s),\dots,x_r)\\
                  +\sum_{s+t=r+1}\sum_{l=1}^s(-1)^{(s-l)(t-1)+s-1+t\sum_{n=1}^{l-1}\abs{x_n}}\sum_{j+k=i}m_{j,s}(x_{1},\dots,m_{k,t}(x_{l},\dots),\dots)=0,
              \end{multline*}
              compare \cite[\S9.2.1 and the formula in the proof of Proposition 9.2.4]{loday_vallette_2012_algebraic_operads}. As noticed in the previous remark, $A$ equipped with the differential $d_0$ and the operations $m_{0,r}$ is a usual $A$-infinity algebra, hence $m_{0,2}$ is a binary product, associative up to the chain homotopy $m_{0,3}$, $d_0$ satisfies the Leibniz rule, etc. Among the new operations, $d_1$ satisfies the Leibniz rule up to the homotopy $\pm m_{1,2}$,
              \begin{multline*}
                  d_1m_{0,2}(x_1,x_2)-m_{0,2}(d_1(x_1),x_2)-(-1)^{\abs{x_1}}m_{0,2}(x_1,d_1(x_2))\\
                  =d_0(-m_{1,2}(x_1,x_2))+m_{1,2}(d_0(x_1),x_2)+(-1)^{\abs{x_1}}m_{1,2}(x_1,d_0(x_2)).
              \end{multline*}
              Our derived $\Oinf{\O{A}}$-algebras differ from Sagave's \cite{sagave_2010_dgalgebras_derived_algebras} in some signs. This is because Sagave works with bicomplexes with commuting differentials and uses the Koszul sign rule with respect to the horizontal and vertical degrees separately, while we use bicomplexes with anti-commuting differentials and apply the Koszul sign rule with respect to the total degree. Both approaches are equivalent by \cite[Remarks 2.2 and 2.8]{muro_roitzheim_2019_homotopy_theory_bicomplexes}.
        \item Recall that a (non-trivial) \emph{$(p,q)$-shuffle}, $p,q\geq1$, is a permutation $\sigma\in\SS_{p+q}$ of the form
              \[\begin{pmatrix}
                      1   & \cdots & p   & p+1 & \cdots & p+q \\
                      u_1 & \cdots & u_p & v_1 & \cdots & v_q
                  \end{pmatrix}\]
              where $u_1<\cdots<u_p$ and $v_1<\dots<v_q$. They form a subset $\SS_{p,q}\subset\SS_{p+q}$. We say that a map $g\colon X^{\otimes^r}\to Y$ \emph{vanishes on shuffles} if
              \[\sum_{\sigma\in\SS_{s,r-s}}(-1)^{\sum_{u_p>v_q}\abs{x_{u_p}}\abs{x_{v_q}}}g(x_{\sigma^{-1}(1)},\dots,x_{\sigma^{-1}(r)})=0\]
              for all $1< s< r$. The exponent of $-1$ is indeed $\alpha_\sigma$.
              A \emph{derived homotopy commutative algebra} is a derived homotopy associative algebra such that the $m_{i,r}$ vanish on shuffles, compare \cite[\S5]{cheng_getzler_2008_transferring_homotopy_commutative}. In particular, the binary product $m_{0,2}$ is commutative.
        \item A map $g\colon X^{\otimes^r}\to Y$ is \emph{skew-symmetric} if
              \[g(x_1,\dots,x_r)=(-1)^{\alpha_\sigma}\operatorname{sign}(\sigma)g(x_{\sigma^{-1}(1)},\dots,x_{\sigma^{-1}(r)})\]
              for any permutation $\sigma\in\SS_r$.
              A \emph{derived homotopy Lie algebra} $L$ is a twisted complex equipped with skew-symmetric bidegree $(-i,r-2+i)$ operations $\ell_{i,r}\colon A^{\otimes^r}\to A$, $r\geq 2$, $i\geq 0$,
              satisfying the following generalized Jacobi identity
              \begin{multline*}
                  \sum_{j+k=i}d_j(\ell_{k,r}(x_1,\dots,x_r))+\sum_{s=1}^r(-1)^{r-1+\sum_{t=1}^{s-1}\abs{x_t}}\sum_{j+k=i}\ell_{j,r}(x_1,\dots,d_k(x_s),\dots,x_r)\\
                  +\sum_{\substack{s+t=r+1\\\sigma\in\SS_{t,s-1}}}(-1)^{(s-1)t+\alpha_\sigma}\operatorname{sign}(\sigma)\sum_{j+k=i}\ell_{j,s}(\ell_{k,t}(x_{u_1},\dots,x_{u_t}),x_{v_1},\dots,x_{v_{s-1}})=0,
              \end{multline*}
              compare \cite{lada_markl_1995_strongly_homotopy_lie}. Here, $\alpha_\sigma$ can be computed as in (2).
              Like in previous cases, $L$ equipped with the differential $d_0$ and the operations $\lambda_{0,r}$ is a usual $L$-infinity algebra. Hence, $\ell_{0,2}$ is a binary skew-symmetric product which satisfies the Jacobi identity up to the chain homotopy $\ell_{0,3}$, $d_0$ satisfies the Leibniz rule, and $d_1$ too up to the skew-symmetric homotopy $\pm\ell_{1,2}$,
              \begin{multline*}
                  d_1\ell_{0,2}(x_1,x_2)-\ell_{0,2}(d_1(x_1),x_2)-(-1)^{\abs{x_1}}\ell_{0,2}(x_1,d_1(x_2))\\
                  =d_0(-\ell_{1,2}(x_1,x_2))+\ell_{1,2}(d_0(x_1),x_2)+(-1)^{\abs{x_1}}\ell_{1,2}(x_1,d_0(x_2)).
              \end{multline*}
        \item A \emph{derived homotopy $\unit$-algebra} is just a twisted complex.
    \end{enumerate}
\end{example}

\begin{definition}\label{derived_homotopy_algebra}
    A \emph{split filtered $\Oinf{\O{O}}$-algebra} $A$ is an $\Oinf{\O{O}}$-algebra in $\chain$ such that:
    \begin{enumerate}
        \item For each $n\in\Z$, the degree $n$ module splits as
              \[A_n=\bigoplus_{p+q=n}A_{p,q}\]
              with $A_{p,q}=0$ if $p<0$. 

        \item If we denote
              \[F_mA_n=\bigoplus_{\substack{p+q=n\\p\leq m}}A_{p,q},\]
              the differential of $A$ satisfies $d(F_mA_n)\subset F_mA_{n-1}$. 

        \item The $\Oinf{\O{O}}$-algebra structure is compatible with the filtration, i.e.~the structure maps (co)restrict to
              \[\K{\O{O}}(r)\otimes F_{m_1}A\otimes\cdots\otimes F_{m_r}A\To F_{m_1+\cdots+m_r}A.\]
    \end{enumerate}
\end{definition}

\begin{remark}
    Derived homotopy algebras are examples of filtered complexes which are exhaustive and bounded below, see Definition
    \ref{filtered_complex}, so we can use Lemma \ref{spectral}.
\end{remark}

\begin{corollary}\label{dha_char_2}
    A derived homotopy $\O{O}$-algebra $A$ is the same as a split filtered $\Oinf{\O{O}}$-algebra.
\end{corollary}

\begin{proof}
    Since $A$ is filtered, the restriction of the differential to $A_{p,q}$ is determined by its components,
    \[d_i\colon A_{p,q}\To A_{p-i,q+i-1},\quad i\geq 0.\]
    There are finitely many for each $p$ since $A_{p-i,q+i-1}=0$ for $i>p$. It is well known that \eqref{twisted_complex_equation} is equivalent to the differential of $A$ squaring to zero.

    Similarly, the restriction to $\K{\O{O}}(r)_s\otimes A_{p_1,q_1}\otimes\cdots\otimes A_{p_r,q_r}$ of the structure maps in Remark \ref{alternative} are determined by their components
    \[\K{\O{O}}(r)_s\otimes A_{p_1,q_1}\otimes\cdots\otimes A_{p_r,q_r} \To A_{p_1+\cdots+p_r-i,s-1+q_1+\cdots+q_r+i},\qquad i\geq0,\]
    since these structure maps are compatible with the filtration of $A$. If we denote these components by
    \[\mu\otimes x_1\otimes\cdots\otimes x_r\mapsto \mu_i(x_1,\dots,x_r),\]
    the equations in Proposition \ref{dha_char} are a mere translation of those in Remark \ref{alternative}. Unlike in the proof of Proposition \ref{dha_char}, no change of variables is needed here.
\end{proof}

\begin{remark}
    Proposition \ref{dha_char} extends to an equivalence of categories if we define morphisms of split filtered $\Oinf{\O{O}}$-algebras $f\colon A\to B$ as $\Oinf{\O{O}}$-algebra morphisms preserving the horizontal and the vertical degree, $f(A_{p,q})\subset B_{p,q}$, $p,q\in\Z$.
\end{remark}

Now, we would like to do the same for $\infty$-morphisms between derived homotopy $\O{O}$-algebras, i.e.~we want to characterize them as filtration preserving $\infty$-morphisms between split filtered $\Oinf{\O{O}}$-algebras. In order to achieve this goal we need a down-to-earth description of $\Oinf{\D{\O{O}}}$-algebra $\infty$-morphisms.

\begin{remark}\label{infinity_morphisms}
    For a graded operad $\O{O}$, an $\infty$-morphism between $\Oinf{\O{O}}$-algebras can be described in terms of the decomposition $\Delta$ and the infinitesimal decomposition $\Delta_{(1)}$ of $\K{\O{O}}$.  We use the Sweedler notations in \eqref{sweedler_decomposition} and \eqref{sweedler_infinitesimal_decomposition} for these $\Delta$ and $\Delta_{(1)}$.

    An \emph{$\infty$-morphism} between $\Oinf{\O{O}}$-algebras $f\colon A\leadsto B$ is given by structure morphisms
    \begin{align*}
        \K{\O{O}}(r)_{n_0}\otimes A_{n_1}\otimes \cdots\otimes A_{n_r} & \To B_{n_0+n_1+\cdots+n_r},       \\
        \mu\otimes x_1\otimes\cdots\otimes x_r                         & \;\mapsto\;f(\mu)(x_1,\dots,x_r),
    \end{align*}
    satisfying the following equations, where we borrow notation from Remark \ref{alternative} for signs:
    \[f(\mu\cdot\sigma)(x_1,\dots,x_r)=(-1)^{\alpha_\sigma} f(\mu)(x_{\sigma^{-1}(1)},\dots,x_{\sigma^{-1}(r)}),\]
    for any permutation $\sigma\in\SS_r$, and
    \begin{multline*}d(f(\mu)(x_1,\dots,x_r))-\sum_{s=1}^r(-1)^{\beta}f(\mu)(x_1,\dots,d(x_s),\dots,x_r)\\
        =\sum_{(\mu)}(-1)^{\gamma}f(\mu^{(1)})(x_{\sigma^{-1}(1)},\dots,\mu^{(2)}(x_{\sigma^{-1}(l)},\dots),\dots)\\
        -\sum_{[\mu]}(-1)^{\lambda}\nu(f(\nu^1)(x_{\tau^{-1}(1)},\dots),\dots,f(\nu^l)(\dots, x_{\tau^{-1}(r)})),
    \end{multline*}
    where $\beta$ and $\gamma$ are as in Remark \ref{alternative} and $\lambda$ consists of adding up $\alpha_\tau$
    and
    \[|\nu^u| |x_{\tau^{-1}(v)}|\]
    whenever $\nu^u$ appears after $x_{\tau^{-1}(v)}$.

    The \emph{underlying morphism} of an $\infty$-morphism $f\colon A\leadsto B$ is $f(1)\colon A\to B$. The \emph{composition} of $f$ with another $\infty$-morphism $g\colon B\leadsto C$ is given by
    \begin{align*}
        (gf)(\mu)(x_1,\dots,x_r) & =\sum_{[\mu]}(-1)^\lambda g(\nu)(f(\nu_1)(x_{\tau^{-1}(1)},\dots),\dots,f(\nu_l)(\dots,x_{\tau^{-1}(r)})).
    \end{align*}
    The category of $\Oinf{\O{O}}$-algebras and $\infty$-morphisms between them will be denoted by $\inftyalgebras{\chain}{\Oinf{\O{O}}}$.

    This description also works in the category of graded complexes using the total degree for signs.
\end{remark}

\begin{definition}\label{derived_infty_morphism}
    For $\O{O}$ a graded Koszul operad, a \emph{derived $\infty$-morphism} between derived homotopy $\O{O}$-algebras is an $\infty$-morphism of $\Oinf{\D{\O{O}}}$-algebras between them.
\end{definition}

\begin{proposition}\label{dhm_char}
    Let $\O{O}$ be a graded Koszul operad.
    An $\infty$-morphism of $\Oinf{\D{\O{O}}}$-algebras $f\colon A\leadsto B$ is the same as a family of module morphisms, $i\geq 0$,
    \begin{align*}
        \K{\O{O}}(r)_s\otimes A_{p_1,q_1}\otimes\cdots\otimes A_{p_r,q_r} & \To B_{p_1+\cdots+p_r-i,s+q_1+\cdots+q_r+i}, \\
        \mu\otimes x_1\otimes\cdots\otimes x_r                            & \;\mapsto\; f_i(\mu)(x_1,\dots,x_r),
    \end{align*}
    satisfying the following equations for $i\geq 0$,
    \[f(\mu\cdot\sigma)_i(x_1,\dots,x_r)=(-1)^{\alpha_\sigma} f(\mu)_i(x_{\sigma^{-1}(1)},\dots,x_{\sigma^{-1}(r)}),\quad\sigma\in\SS_r,\]
    \begin{multline*}
        \sum_{j+k=i}d_j(f_k(\mu)(x_1,\dots,x_r))-\sum_{s=1}^r(-1)^{\beta}\sum_{j+k=i}f_j(\mu)(x_1,\dots,d_k(x_s),\dots,x_r)
        \\
        =\sum_{(\mu)}(-1)^{\gamma}\sum_{j+k=i}f_j(\mu^{(1)})(x_{\sigma^{-1}(1)},\dots,\mu_k^{(2)}(x_{\sigma^{-1}(l)},\dots),\dots)\\
        -\sum_{[\mu]}(-1)^{\lambda}\sum_{j+k_1+\cdots+k_l=i}\nu_j(f_{k_1}(\nu^1)(x_{\tau^{-1}(1)},\dots),\dots,f_{k_l}(\nu^l)(\dots, x_{\tau^{-1}(r)})).
    \end{multline*}
    Here, $\alpha,\beta,\gamma$ are as in \eqref{alfa} and \eqref{beta_gamma}.
\end{proposition}

The proof is analogous to that of Proposition \ref{dha_char}, hence we skip it.

\begin{definition}
    Given a graded Koszul operad $\O{O}$ and two split filtered $\Oinf{\O{O}}$-algebras $A,B$ we say that an $\infty$-morphism $f\colon A\leadsto B$ is \emph{filtration preserving} if $f(F_mA)\subset F_mB$, $m\in\Z$.
\end{definition}

\begin{corollary}\label{dhm_agrees}
    Let $\O{O}$ be a graded Koszul operad.
    A derived $\infty$-morphism of derived homotopy $\O{O}$-algebras is the same as a filtration-preserving $\infty$-morphism between split filtered $\Oinf{\O{O}}$-algebras.
\end{corollary}

This is analogous to Corollary \ref{dha_char} so we also skip it. This correspondence is compatible with composition of (derived) $\infty$-morphisms, so it defines an equivalence of categories.

\begin{example}\label{examples_derived_infty-morphisms}
    In this remark, we explicitly describe derived $\infty$-morphisms for the operads in Example \ref{examples_derived_homotopy_algebras}.
    \begin{enumerate}
        \item A derived $\infty$-morphism of derived homotopy associative algebras $f\colon A\leadsto B$ consists of bidegree $(-i,r-1+i)$ maps $f_{i,r}\colon A^{\otimes^r}\to B$, $r\geq 1$, $i\geq 0$,
              satisfying
              \begin{multline*}
                  \sum_{j+k=i}d_j(f_{k,r}(x_1,\dots,x_r))-\sum_{s=1}^r(-1)^{r-1+\sum_{t=1}^{s-1}\abs{x_t}}\sum_{j+k=i}f_{j,r}(x_1,\dots,d_k(x_s),\dots,x_r)\\
                  =\sum_{s+t=r+1}\sum_{l=1}^s(-1)^{s-1+(s-l)(t-1)+t\sum_{n=1}^{l-1}\abs{x_n}}\sum_{j+k=i}f_{j,s}(x_{1},\dots,m_{k,t}(x_{l},\dots),\dots)\\
                  -\sum_{s=2}^r\hspace{-2pt}\sum_{\sum\limits_{l=1}^st_l=r}\hspace{-5pt}(-1)^{\sum\limits_{l=1}^s\big((t_l+1)(s-l)+(t_l-1)\hspace{-5pt}\sum\limits_{u=1}^{\sum\limits_{v=1}^{l-1}t_v}\hspace{-5pt}\abs{x_u}\big)}\hspace{-20pt}\sum_{j+k_1+\cdots+k_s=i}\hspace{-20pt}m_{j,s}(f_{k_1,t_1}(x_{1},\dots),\dots,f_{k_s,t_s}(\dots, x_{r})).
              \end{multline*}
              The morphism $f_{0,1}\colon A\to B$ is a chain map with respect to the differential $d_0$. It preserves the product up to the chain homotopy $f_{0,2}$. Moreover, $f_{0,1}$ commutes with the differential $d_1$ up to a chain homotopy defined from $f_{1,1}$, etc. These derived $\infty$-morphisms coincide with Sagave's maps between derived $\Oinf{\O{A}}$-algebras \cite{sagave_2010_dgalgebras_derived_algebras} up to the changes of sign conventions explained in Example \ref{examples_derived_homotopy_algebras} (1).
        \item A derived $\infty$-morphism of derived homotopy commutative algebras $f\colon A\leadsto B$ is a derived $\infty$-morphism of derived homotopy associative algebras such that the $f_{i,r}$ vanish on shuffles.
        \item A derived $\infty$-morphism of derived homotopy Lie algebras $f\colon L\leadsto L'$ consists of skew-symmetric bidegree $(-i,r-1+i)$ maps $f_{i,r}\colon L^{\otimes^r}\to L'$, $r\geq 1$, $i\geq 0$,
              satisfying
              \begin{multline*}
                  \sum_{j+k=i}d_j(f_{k,r}(x_1,\dots,x_r))-\sum_{j+k=i}(-1)^{r-1+\sum_{t=1}^{s-1}\abs{x_t}}f_{j,r}(x_1,\dots,d_k(x_s),\dots,x_r)\\
                  =\sum_{\substack{s+t=r+1\\\sigma\in\SS_{t,s-1}}}(-1)^{(s-1)t+\alpha_\sigma}\operatorname{sign}(\sigma)\sum_{j+k=i}f_{j,s}(\ell_{k,t}(x_{u_1},\dots,x_{u_t}),x_{v_1},\dots,x_{v_{s-1}})\\
                  -\hspace{-25pt}\sum_{\substack{2\leq s\leq n\\ t_1+\cdots+t_s=r\\\tau\in\SS_{t_1,\dots,t_s}\\w_{l,1}<w_{l+1,1}\text{ if }t_l=t_{l+1}}}\hspace{-25pt}(-1)^{\frac{s(s-1)}{2}+\sum\limits_{l=1}^{s-1}t_l(s-l)+\alpha_\tau+\lambda}\operatorname{sign}(\tau)\hspace{-10pt}\sum_{j+k_1+\cdots+k_s=i}\hspace{-10pt}\ell_{j,s}(f_{k_1,t_1}(x_{w_{1,1}},\dots),\dots\\
                  \dots,f_{k_s,t_s}(\dots,x_{w_{s,t_s}})).
              \end{multline*}
              A \emph{$(t_1,\dots,t_s)$-shuffle}, $t_l\geq 1$, is a permutation of the form
              \[\begin{pmatrix}
                      1       & \cdots & t_1       & \cdots\cdots & \sum_{l=1}^{s-1}t_l+1 & \cdots & \sum_{l=1}^{s}t_l \\
                      w_{1,1} & \cdots & w_{1,t_1} & \cdots\cdots & w_{s,1}               & \cdots & w_{s,t_s}
                  \end{pmatrix}\]
              with $w_{l,1}<\cdots <w_{l,t_l}$ for each $l$. This is the obvious generalization of $(p,q)$-shuffles. The set $\SS_{t_1,\dots,t_s}$ in the last summation is the set of $(t_1,\dots,t_s)$-shuffles such that, if $t_l=t_{l+1}$ then $w_{l,1}<w_{l+1,1}$.
              Moreover, $\lambda$ consists of adding $(t_u-1)\abs{x_{w_{p,q}}}$ for all $p<u$. Compare \cite[Definition 2.3]{allocca_2014_homomorphisms_infty_modules}.
        \item A derived $\infty$-morphism of derived homotopy $\unit$-algebras is a \emph{twisted morphism} $f\colon X\leadsto Y$ between twisted complexes. It consists of a family of module morphisms
              \[f_i\colon X_{p,q}\To X_{p-i,q+i},\quad i\geq 0,\]
              satisfying
              \[\sum_{j+k=i}f_jd_k=\sum_{j+k=i}d_jf_k,\quad i\geq 0.\]
              The composition of $f$ with another twisted morphism $g\colon B\leadsto C$ is given by
              \[(gf)_i=\sum_{j+k=i}g_jf_k.\]
    \end{enumerate}
\end{example}

\begin{remark}
    For $\O{O}$ graded Koszul, a derived $\infty$-morphism of derived $\Oinf{\O{O}}$-algebras $f\colon A\leadsto B$ has an \emph{underlying twisted morphism} $f(1)\colon A\leadsto B$ given by
    \[f(1)_i\colon A_{p,q}\To B_{p-i,q+i},\quad i\geq0.\]
\end{remark}

\section{Minimal models}\label{section:minimal_models}

Throughout this section $\O{O}=\P{E}{R}$ stands for a graded reduced quadratic Koszul operad which, arity-wise, is a projective module over the corresponding symmetric group. Hence we can apply Theorem \ref{main_main}.

\begin{definition}\label{E2-equivalence}
    A derived $\infty$-morphism between derived homotopy $\O{O}$-algebras $f\colon A\leadsto B$ is an \emph{$E_2$-equivalence} if, regarded as a filtration-preserving $\infty$-morphism between split filtered $\Oinf{\O{O}}$-algebras (see Corollary \ref{dhm_agrees}), it induces an isomorphism between the $E^2$ pages of the spectral sequences associated to the filtrations.
\end{definition}

A \textit{differential graded $\O{O}$-algebra} is an $\O{O}$-algebra in $\chain$.

\begin{definition}\label{minimal_models}
    A twisted complex $X$ is \emph{minimal} if it has trivial vertical differential $d_0=0$. A $\Oinf{\D{\O{O}}}$-algebra is \emph{minimal} if its underlying twisted complex is minimal (see Proposition \ref{dha_char}).

    A \emph{minimal model} for a differential graded $\O{O}$-algebra $A$ is an $E^2$-equivalence $f\colon \minimal{A}\leadsto A$ whose source is a $\kk$-projective minimal derived $\Oinf{\O{O}}$-algebra.
\end{definition}

\begin{remark}\label{minimal_remark_1}
    Since $E^2$-equivalences are the inverse image of isomorphisms by a functor, they are closed under composition and, moreover, they satisfy the 2-out-of-3 property.

    The homology of a graded complex $X$ is denoted by $H_*(X)$. If $A$ is a $\Oinf{\D{\O{O}}}$-algebra, then $H_*(A)$ is a $\D{\O{O}}$-algebra since $H_*(\Oinf{\D{\O{O}}})=\D{\O{O}}$ by Koszulity.

    If $A$ is minimal then $H_*(A)=A$, hence it also carries an underlying $\D{\O{O}}$-algebra structure. Regarded as an $\O{O}$-algebra in $\bichainu$ (see Proposition \ref{dO-algebra}), the vertical differential is trivial $d_v=0$, the horizontal differential is $d_h=d_1$, and given a generator $\mu\in E$ of $\O{O}$
    \[\mu(x_1,\dots,x_r)=(s\mu)_0(x_1,\dots,x_r).\] Here, the right hand side is part of the $\Oinf{\D{\O{O}}}$-algebra structure as described in Proposition \ref{dha_char}, and the left hand side is the $\O{O}$-algebra structure which is part of the underlying $\D{\O{O}}$-algebra structure

    If $f\colon A\leadsto B$ is an $\Oinf{\D{\O{O}}}$-algebra $\infty$-morphism, then the map $f(1)_0\colon A\to B$ in $\grchain$ induces a $\D{\O{O}}$-algebra morphism $H_*(A)\to H_*(B)$. In particular, if $A$ and $B$ are minimal $f(1)_0\colon A\to B$ is a morphism between the previous underlying $\D{\O{O}}$-algebra structures.
\end{remark}

\begin{proposition}\label{minimal_exist}
    Any differential graded $\O{O}$-algebra $A$ has a minimal model.
\end{proposition}

\begin{proof}
    Any cofibrant object $X$ in $\chain$ (endowed with the projective model structure) with degree-wise projective homology $H_*(X)$ is chain homotopy equivalent to $H_*(X)$ endowed with the trivial differential. Indeed, $H_*(X)$ is also cofibrant in $\chain$ under these assumptions and all objects in $\chain$ are fibrant. Therefore, any quasi-isomorphism $H_*(X)\to X$ is a homotopy equivalence. Such a quasi-isomorphism (a cycle selection map) can be constructed because $H_*(X)$ is projective. Moreover, any homotopy equivalence $H_*(X)\to X$ is the inclusion of a deformation retract because chain homotopic endomorphisms of $H_*(X)$ are equal since $H_*(X)$ has trivial differential. We can actually make it part of a contraction in the sense of \cite[Definition 2.1 and Remark 2.1]{berglund_2014_homological_perturbation_theory}.
    This obviously extends to the category of graded complexes $\grchain=\chain^\Z$ with the product model structure.

    Let $\cofres{A}\to A$ be a cofibrant resolution in the semi-model category $\algebras{\bichainCE}{\O{O}}$ of Proposition \ref{semi-model}. The bicomplex underlying $\cofres{A}$ is a Cartan--Eilenberg resolution, see Remark \ref{is_cartan_eilenberg}, so $\cofres{A}$ has projective vertical homology. We simply denote the vertical homology by $H_*(\cofres{A})$, since it is the homology of the underlying graded complex.

    The graded complex underlying $\cofres{A}$ is cofibrant because the graded complex morphisms underlying the generating cofibrations in $\bichainCE$ are cofibrations in $\grchain$, see \cite[\S4]{muro_roitzheim_2019_homotopy_theory_bicomplexes}, and colimits are computed point-wise both in $\bichain$ and $\grchain$. Therefore, we can take a map $i\colon H_*(\cofres{A})\to \cofres{A}$ in $\grchain$ which is part of a contraction. We can now invoke the homotopy transfer theorem, see \cite[Theorem 10.3.1]{loday_vallette_2012_algebraic_operads} and \cite[Theorem 1.3]{berglund_2014_homological_perturbation_theory}. This theorem endows $H_*(\cofres{A})$ with a minimal derived homotopy $\O{O}$-algebra structure $\minimal{A}$ and enhances the map of graded complexes $i\colon H_*(\cofres{A})\to \cofres{A}$ to a derived $\infty$-morphism $i\colon\minimal{A}\leadsto \cofres{A}$. This derived $\infty$-morphism is an $E^2$-equivalence. It actually induces an isomorphism on the $E^1$ page of the corresponding spectral sequences, since this term is the vertical homology. If we compose it with $\cofres{A}\to A$ we obtain the desired minimal model $\minimal{A}\leadsto A$.

    The homotopy transfer theorem is purely combinatorial and does not depend on the underlying ground ring. We will however need further hypotheses later when we want to find an left inverse for $i\colon\minimal{A}\leadsto A$. Such a left inverse will be needed to ensure the essential uniqueness of minimal models.
\end{proof}

\begin{remark}\label{minimal_remark}
    The proof of Proposition \ref{minimal_exist} provides a method for the construction of minimal models.
    Cofibrant resolutions in $\algebras{\bichainCE}{\O{O}}$ are usually huge because their underlying bigraded $\O{O}$-algebras are free. Nevertheless, in practice, it is often possible to find an $E^2$-equivalence $\tilde{A}\to A$ in $\algebras{\bichainCE}{\O{O}}$ with $\tilde{A}$ smaller than any $\cofres{A}$ whose homology $H_*(\tilde{A})$ is $\kk$-projective and $\tilde{A}$ contracts onto $H_*(\tilde{A})$, and this suffices to enhance $H_*(\tilde{A})$ to a minimal model. This is how we obtain the minimal models in the examples below.
\end{remark}

\begin{example}\label{example_dugger_shipley}
    For $p\in\Z=\kk$ a prime, Dugger and Shipley consider in \cite{dugger_shipley_2009_curious_example_triangulatedequivalent} the differential graded unital associative algebra
    \[A=\frac{\Z\langle e,x^{\pm1}\rangle}{(e^2,ex+xe-x^2)},\qquad \abs{e}=\abs{x}=1,\]
    with differential
    \begin{align*}
        d(e) & =p, & and d(x) & =0.
    \end{align*}
    The homology is
    \[H_*(A)=\Z/(p)\langle x^{\pm1}\rangle.\]

    Consider the unital associative algebra in $\bichain$ with underlying bigraded algebra
    \[\minimal{A}=\frac{\Z\langle x^{\pm1}, c\rangle}{(c^2,cx+xc)},\qquad \abs{x}_b=(0,1),\quad\abs{c}_b=(1,0).\]
    The vertical differential is trivial $d_v=0$ and the horizontal differential is given by
    \begin{align*}
        d_h(c) & =-p, &
        d_h(x) & =0.
    \end{align*}
    Since $c$ is a square-zero element in the center, it is easy to see that the horizontal homology of $\minimal{A}$ coincides with $H_*(A)$.

    We can extend $\minimal{A}$ to a (minimal) derived homotopy $\O{A}$-algebra whose higher operations vanish except for $m_{1,2}$, i.e.~$d_1=d_h$ is the horizontal differential, $m_{0,2}$ is the binary associative product, $d_i=0$ for $i\neq 1$, and $m_{i,r}=0$ for $(i,r)\neq (0,2),(1,2)$. Moreover, $m_{1,2}$ is given by
    \begin{align*}
        m_{1,2}(x^i,x^j)   & =0, \quad i,j\in\Z;                            \\
        m_{1,2}(cx^i,x^j)  & =0, \quad i,j\in\Z;                            \\
        m_{1,2}(x^i,cx^j)  & =\left\lbrace\begin{array}{ll}
            0,         & i\text{ even}, \\
            x^{i+j+1}, & i\text{ odd};
        \end{array}\right.  \\
        m_{1,2}(cx^i,cx^j) & =\left\lbrace\begin{array}{ll}
            0,          & i\text{ even}, \\
            cx^{i+j+1}, & i\text{ odd}.
        \end{array}\right.
    \end{align*}

    We also consider the map $f\colon \minimal{A}\to A$ of bigraded unital associative algebras given by $f(x)=x$ and $f(c)=0$. This map is not compatible with horizontal differentials, but we can extend it to a derived $\infty$-morphism $f\colon \minimal{A}\leadsto A$ of derived homotopy $\O{A}$-algebras whose only non-trivial pieces are $f_{0,1}$ (the map $\minimal{A}\to A$ above) and $f_{1,1}$, i.e.~$f_{i,n}=0$ for $(i,n)=(0,1),(1,1)$, and
    \begin{align*}
        f_{1,1}(x^n)  & =0;    \\
        f_{1,1}(cx^n) & =ex^n.
    \end{align*}
    We do not include here an explicit computation showing that the derived $\infty$-morphism equations hold. However, it is clear that $f$ is an $E^2$-equivalence because of the previous computation of $H_*(A)$ and the horizontal homology of $\minimal{A}$.

    We would also like to indicate how we have obtained this minimal model by the method described in the proof of Proposition \ref{minimal_exist} and in Remark \ref{minimal_remark}.

    We consider the unital associative algebra in $\bichain$ with underlying bigraded algebra
    \[\tilde{A}=\frac{\Z\langle x^{\pm1},a,b,c\rangle}{(a^2,c^2,ax+xa-x^2,ba-ab,bx-xb,cx+xc,ac+ca,bc-cb)}.\]
    The generators have bidegrees
    \begin{align*}
        \abs{x}_b=\abs{a}_b & =(0,1), &
        \abs{b}_b           & =(0,0), &
        \abs{c}_c           & =(1,0).
    \end{align*}
    The first two relations are nilpotency relations, the third one is a twisted commutativity relation, and the rest are plain commutativity relations (the middle signs correspond to the Koszul sign rule). The bicomplex differentials are defined by
    \begin{align*}
        d_v(x) & =0,    &
        d_v(a) & =b,    &
        d_v(b) & =0,    &
        d_v(c) & =0,    &
        d_h(c) & = b-p.
    \end{align*}
    The horizontal differential of horizontal degree $0$ elements must be trivial for degree reasons. The relations $d_v^2=0=d_h^2$ and $d_vd_h+d_hd_v=0$ are clear on generators.

    We define the map $\tilde{A}\to A$ of unital associative algebras in $\bichain$ as
    \begin{align*}
        x & \mapsto x, &
        a & \mapsto e, &
        b & \mapsto p, &
        c & \mapsto 0.
    \end{align*}

    The vertical homology $H_*(\tilde{A})$ is $\minimal{A}$ above as a unital associative algebra in $\bichain$.
    We have a contraction in $\grchain$
    \begin{center}
        \begin{tikzcd}
            H_*(\tilde{A})=\minimal{A}\ar[r, yshift=2, "i"]&\tilde{A}\ar[l, yshift=-2, "g"]\ar[loop right, "h"]
        \end{tikzcd}
    \end{center}
    where $i$, as a map of unital associative bigraded algebras, is the obvious inclusion. Beware that this algebra map is not compatible with horizontal differentials (it is compatible with vertical differentials though). A $\kk$-basis of $\minimal{A}$ is $\{x^n,cx^n\}_{n\in\Z}$ and a $\kk$-basis of $\tilde{A}$ is $\{b^kx^n,b^kcx^n,ab^kx^n,ab^kcx^n\}_{n\in\Z,k\geq 0}$.
    The retraction $g$ is the identity on the basis elements $x^n,cx^n$, $n\in\Z$, and trivial on the rest, and the homotopy $h$ is given by
    \begin{align*}
        h(x^n)=h(cx^n)                 & =0;                      \\
        h(b^kx^n)                      & =ab^{k-1}x^n,\quad k>0;  \\
        h(b^kcx^n)                     & =ab^{k-1}cx^n,\quad k>0; \\
        h(ab^kx^{n-1})=h(ab^kcx^{n-1}) & =0,\quad k\geq 0.
    \end{align*}
    A direct application of the homotopy transfer theorem for the operad $\D{\O{A}}$ yields the previous minimal model.

    Since we are not carrying out all computations step by step, the $\O{A}$-algebra $\tilde{A}$ may seem counter-intuitive. However, the idea behind it is very easy. The complex underlying $A$ is a direct sum of shifted copies of
    \[\cdots\to 0\to {}^e\Z\stackrel{p}{\To}{}^1\Z\to 0\to \cdots.\]
    Here, the left superscripts indicate the generators of one of the copies, the rest are obtained by multiplication with $x^n$, $n\in \Z$.
    The smallest Cartan--Eilenberg resolution of this complex in $\bichain$ is
    \begin{equation}\label{a_cartan-eilenberg_resolution}
        \begin{tikzcd}[execute at end picture={\draw[-, ultra thick, opacity = .5] ($(A)+(1,.3)$) to ($(B)+(1,-.2)$);}]
            |[alias=A]|\vdots\ar[d]&\vdots\ar[d]&\vdots\ar[d]\\
            {}^e\Z\ar[d,"p"']&{}^a\Z\ar[d,"\left(\begin{smallmatrix}
                    0\\1
                \end{smallmatrix}\right)"]\ar[l,near start, "1"']&0\ar[d]\ar[l]&\cdots\ar[l]\\
            {}^1\Z\ar[d]&{}^1\Z\oplus{}^b\Z\ar[d]\ar[l,near end, "{(1,p)}"]&{}^c\Z\ar[d]\ar[l,"\left(\begin{smallmatrix}
                    -p\\1
                \end{smallmatrix}\right)"]&\cdots\ar[l]\\
            |[alias=B]|\vdots&\vdots&\vdots\\
        \end{tikzcd}
    \end{equation}
    The complex is to the left of the vertical line, and the Cartan--Eilenberg resolution is to the right. The non-depicted part is trivial. The $\O{A}$-algebra $\tilde{A}$ has been obtained by making an associative algebra out of this, incorporating the unit $x$, and imposing as many relations as possible to make the underlying bicomplex similar to shifted copies of this Cartan--Eilenberg resolution.

    \begin{example}\label{commutative_minimal}
        Let $\kk=\Q[t]$ and let $\O{C}$ be the (non-unital) commutative operad. We denote by $S$ the free symmetric algebra functor and consider the symmetric graded algebra
        \[A=\frac{S(x,y,x_t,y_t)}{(x^2,y^2,xy,xx_t,yy_t,x_ty)},\qquad\abs{x}=\abs{y}=2,\quad\abs{x_t}=\abs{y_t}=3,\]
        endowed with the differential defined by
        \begin{align*}
            d(x)   & =d(y)=0, &
            d(x_t) & =tx,     &
            d(y_t) & =ty.
        \end{align*}
        The compatibility of the differential with the relations is easy to check.
        A $\kk$-linear basis of $A$ is
        $\{x,y,x_t,y_t,xy_t,x_ty_t\}$.
        Hence $H_*(A)$ is a $\Q$-vector space with basis
        $\{x,y,xy_t\}$
        and trivial product.

        In this case, we have a minimal model with underlying free bigraded $\kk$-module
        \begin{align*}
            \minimal{A}         & =\kk\cdot\{x,y,z,c_x,c_y,c_z\}, &
            \abs{x}_b=\abs{y}_b & =(0,2),                         &
            \abs{z}_b           & =(0,5),                                                               \\
                                &                                 & \abs{c_x}_b=\abs{c_y}_b & =(1,2), &
            \abs{c_z}_b         & =(1,5),
        \end{align*}
        trivial vertical differential $d_v=0$, and horizontal differential
        \begin{align*}
            d_h(x)=d_h(y)=d_h(z) & =0,   &
            d_h(c_x)             & =-tx, &
            d_h(c_y)             & =-ty, &
            d_h(c_z)             & =-tz.
        \end{align*}
        All derived $\Oinf{\O{C}}$-algebra operations in $\minimal{A}$ are trivial except for $d_1=d_h$ and $m_{1,2}$ (even the product $m_{0,2}$ is trivial). The operation $m_{1,2}$ is determined by the formulas
        \begin{align*}
            m_{1,2}(x,c_x)   & =0,   &
            m_{1,2}(x,c_y)   & =z,   &
            m_{1,2}(c_x,y)   & =0,   &
            m_{1,2}(y,c_y)   & =0,     \\
            m_{1,2}(c_x,c_x) & =0,   &
            m_{1,2}(c_x,c_y) & =c_z, &
            m_{1,2}(c_y,c_y) & =0.
        \end{align*}
        The remaining cases either follow from the fact that $m_{1,2}$ is skew-symmetric, since it is a binary operation which  vanishes on shuffles, or they vanish for degree reasons.

        The maps $f_{i,r}$ defining the $E^2$-equivalence $f\colon\minimal{A}\leadsto A$ are trivial except for $f_{0,1}$ and $f_{1,1}$, which are given by
        \begin{align*}
            f_{0,1}(x)   & =x,      &
            f_{0,1}(y)   & =y,      &
            f_{0,1}(z)   & =xy_t,     \\
            f_{0,1}(c_x) & =x_t,    &
            f_{0,1}(c_y) & =y_t,    &
            f_{0,1}(c_z) & =x_ty_t.
        \end{align*}
        The remaining cases vanish for degree reasons.

        For this computation, we have used the version of the Cartan--Eilenberg resolution \eqref{a_cartan-eilenberg_resolution} for $\Q[t]$, consisting of replacing $\Z$ with $\Q[t]$ and $p$ with $t$. With this in mind, we can consider the bigraded commutative algebra $\tilde{A}$ with generators
        \begin{equation}\label{commutative_generators}
            \underbrace{u_x,b_x,u_y,b_y}_{(2,0)},\underbrace{a_x,a_y}_{(3,0)},\underbrace{c_x,c_y}_{(1,2)},
        \end{equation}
        with the indicated bidegree. The relations are that all binary products vanish except for
        \begin{equation}\label{commutative_non_relations}
            u_xa_y,b_xa_y,c_xa_y,a_xa_y
        \end{equation}
        and their symmetrics. The map $\tilde{A}\to A$ is in this case given by
        \begin{align*}
            u_x & \mapsto x,   &
            b_x & \mapsto tx,  &
            a_x & \mapsto x_t, &
            u_y & \mapsto y,   &
            b_y & \mapsto ty,  &
            a_y & \mapsto y_t,
        \end{align*}
        and zero on $c_x,c_y$.

        A $\kk$-linear basis of $\tilde{A}$ is given by the elements \eqref{commutative_generators} and \eqref{commutative_non_relations}. Hence, the underlying bicomplex of $\tilde{A}$ consists of three shifted copies of \eqref{a_cartan-eilenberg_resolution} for $\Q[t]$ and the vertical homology $H_*(\tilde{A})$ is $\minimal{A}$ above as a commutative algebra in $\bichain$.

        We have a contraction in $\grchain$
        \begin{center}
            \begin{tikzcd}
                H_*(\tilde{A})=\minimal{A}\ar[r, yshift=2, "i"]&\tilde{A}\ar[l, yshift=-2, "g"]\ar[loop right, "h"]
            \end{tikzcd}
        \end{center}
        with
        \begin{align*}
            i(x)   & =u_x,    &
            i(y)   & =u_y,    &
            i(z)   & =u_xa_y, &
            i(c_x) & =c_x,    &
            i(c_y) & =c_y,    &
            i(c_z) & =c_xa_y.
        \end{align*}
        Note that $i$ preserves products. The retraction $g$ is the identity on the image of $i$ and trivial on the rest of the basis, and the homotopy $h$ is given by the formulas
        \begin{align*}
            h(b_x)    & =a_x,    &
            h(b_y)    & =a_y,    &
            h(b_xa_y) & =a_xa_y, &
            hi        & =0,      &
            h^2       & =0.
        \end{align*}
        The minimal model $\minimal{A}$ is obtained by direct application of the homotopy transfer theorem for the operad $\D{\O{C}}$ to this contraction.
    \end{example}

    \begin{example}\label{Lie_minimal}
        Let $\kk=\Q[t]$ again and let $\O{L}$ be the Lie operad. We consider the graded Lie algebra
        generated by $x,y,x_t,y_t$ in degrees $\abs{x}=\abs{y}=2$, $\abs{x_t}=\abs{y_t}=3$ with relations all possible triple brackets and all binary brackets except for $[x,y_t],[x_t,y_t]$ and their symmetrics.
        A $\kk$-linear basis of $A$ is therefore
        $\{x,y,x_t,y_t,[x,y_t],[x_t,y_t]\}$, and $H_*(A)$ is a $\Q$-vector space with basis
        $\{x,y,[x,y_t]\}$
        and trivial bracket.

        The same minimal model $\minimal{A}$ as in Example \ref{commutative_minimal} works in this case, with trivial Lie bracket and $\ell_{1,2}$ instead of $m_{1,2}$ in this case. The $E^2$-equivalence $f\colon \minimal{A}\leadsto A$ is also defined by the same formulas.
    \end{example}

\end{example}

\section{The (co)bar construction for derived homotopy algebras}\label{cobar_section}

In this section $\O{O}=\P{E}{R}$ is a graded reduced quadratic Koszul operad. Moreover, either $\O{O}$ is non-symmetric (or rather the symmetrization of a non-symmetric operad) or $\Q\subset\kk$. In particular $\O{O}$ and $\K{\O{O}}$ are arity-wise projective over the corresponding symmetric group and we can apply Theorem \ref{main_main}.

We have already reached our goal of constructing minimal models for differential graded $\O{O}$-algebras, regardless of the ground ring. These minimal models are derived homotopy $\O{O}$-algebras. Now we would like to show that minimal models are essentially unique and, moreover, that they can be strictified so that we can recover the original differential graded $\O{O}$-algebra up to quasi-isomorphism. In the non-derived case, this is done by means of the bar-cobar adjunction. We will use a modified bar-cobar adjunction for $\D{\O{O}}$- and $\Oinf{\D{\O{O}}}$-algebras. In the derived setting, the problem is that if we start with a derived homotopy $\O{O}$-algebra and apply the bar and the cobar constructions, the resulting $\D{\O{O}}$-algebra is not bounded. We have to truncate it in a clever way so as not to change its homotopy type.

We borrow from \cite[\S0.2]{fresse_2009_operadic_cobar_constructions} the concept of perturbed complex, called twisted complex therein, but this has a different meaning here.

\begin{definition}
    A \emph{perturbed complex} is a complex $X$ with differential $d$ and a degree $-1$ endomorphism $\partial\colon X\to X$ of the underlying graded module, often called \emph{perturbation}, such that $(d+\partial)^2=0$, so the underlying graded module of $X$ with the new differential $d+\partial$ is another complex that we denote by $(X,\partial)$.
\end{definition}

\begin{remark}
    This notion extends from $\chain$ to $\grchain$ and to the corresponding categories of $\SS$-modules in the obvious way. Since $d^2=0$, $(d+\partial)^2=0$ is equivalent to $d\partial+\partial d+\partial^2=0$. Quite often $\partial$ is itself a differential, $\partial^2=0$, so the previous equation reduces to $d\partial+\partial d=0$. Some of our perturbed complexes will have additional algebraic structure, hence we will talk about perturbed  (co)algebras, etc.
\end{remark}

\begin{definition}
    A \emph{quasi-cofree $\K{\D{\O{O}}}$-coalgebra} in $\grchain$ is a perturbed $\K{\D{\O{O}}}$-coalgebra of the form $(\K{\D{\O{O}}}\circ X,\partial)$, where $X$ is a graded complex and $\K{\D{\O{O}}}\circ X$ is a cofree $\K{\D{\O{O}}}$-coalgebra (i.e.~the structure maps are given by those of the cooperad $\K{\D{\O{O}}}$).
\end{definition}

The conditions on $\partial$ for $(\K{\D{\O{O}}}\circ X,\partial)$ to be a $\K{\D{\O{O}}}$-coalgebra can be found in \cite[Propostion 4.1.4]{fresse_2009_operadic_cobar_constructions}. Moreover, by \cite[Propostion 4.1.5]{fresse_2009_operadic_cobar_constructions} or \cite[Proposition 11.4.1]{loday_vallette_2012_algebraic_operads}, there is an equivalence
\begin{equation}\label{dO_infinity_bar}
    \bbar_{\D{\O{O}}}\colon\inftyalgebras{\grchain}{\Oinf{\D{\O{O}}}}\To\coalgebrasqf{\grchain}{\K{\D{\O{O}}}}
\end{equation}
from the category of $\Oinf{\D{\O{O}}}$-algebras and $\infty$-morphisms to the category of quasi-cofree $\K{\D{\O{O}}}$-coalgebras. This equivalence is given by $\bbar_{\D{\O{O}}}A=(\K{\D{\O{O}}}\circ A,\partial^A)$ for a certain perturbation $\partial^A$.

Let us describe the functor $\bbar_{\D{\O{O}}}$ on the category $\inftyalgebras{}{\Oinf{\D{\O{O}}}}$ of derived homotopy $\O{O}$-algebras and derived $\infty$-morphisms between them, which is a full subcategory of $\inftyalgebras{\grchain}{\Oinf{\D{\O{O}}}}$ by definition.

We split the Sweedler notation in for the infinitesimal decomposition of $\O{O}$ in \eqref{sweedler_infinitesimal_decomposition} as follows,
\[\Delta_{(1)}(\mu)=\sum_{s=1}^r\mu\circ_s1+\sum_{(\mu)''}(\mu^{(1)}\circ_{l}\mu^{(2)})\cdot\sigma.\]
Here $\mu\in\O{O}(r)$ and $(\mu)''$ stands for the summands where $\mu^{(2)}$ has positive weight.

If $A$ is a derived homotopy $\O{O}$-algebra, then $\partial^A$ is given by the following formula,
\begin{multline}\label{bar_algebras_perturbation}
    \partial^A((\delta^i;\mu);x_1,\dots,x_r)=
    \sum_{s=1}^r(-1)^{\beta}\sum_{\substack{j+k=i\\k>0}}((\delta^j;\mu);x_1,\dots,d_k(x_s),\dots,x_r)\\
    +\sum_{(\mu)''}(-1)^{\gamma}\sum_{j+k=i}((\delta^j;\mu^{(1)}),x_{\sigma^{-1}(1)},\dots,\mu_k^{(2)}(x_{\sigma^{-1}(l)},\dots),\dots).
\end{multline}
%
%
Here $\beta$ and $\gamma$ are as in \eqref{beta_gamma}.

If $A$ is just a derived $\O{O}$-algebra, formula \eqref{bar_algebras_perturbation} simplifies to
\begin{multline}\label{bar_algebras_perturbation_strict}
    \partial^A((\delta^i;\mu);x_1,\dots,x_r)=
    \sum_{s=1}^r(-1)^{\beta}((\delta^{i-1};\mu);x_1,\dots,d_h(x_s),\dots,x_r)\\
    +\sum_{(\mu)''}(-1)^{\gamma}((\delta^i;\mu^{(1)});x_{\sigma^{-1}(1)},\dots,\kappa(\mu^{(2)})(x_{\sigma^{-1}(l)},\dots),\dots).
\end{multline}
Here $\kappa\colon\K{\O{O}}\to\O{O}$ is the canonical twisting morphism.

Given an $\infty$-morphism $f\colon A\leadsto B$ between derived homotopy $\O{O}$-algebras, the induced morphism $\bbar_{\D{\O{O}}}(f)\colon \bbar_{\D{\O{O}}}(A)\to \bbar_{\D{\O{O}}}(B)$ is defined as
\begin{multline}\label{bar_algebras_morphism}
    \bbar_{\D{\O{O}}}(f)((\delta^i;\mu);x_1,\dots,x_r)=\\
    \sum_{[\mu]}\sum_{j+k_1+\cdots+k_l=i}(-1)^{\gamma}((\delta^j;\nu);f(\nu^{1}_{k_1})(x_{\tau^{-1}(1)},\dots),\dots,f(\nu^{l}_{k_l})(\dots, x_{\tau^{-1}(r)})).
\end{multline}
Here we use the Sweedler formula for the decomposition of $\K{\D{\O{O}}}$ in \eqref{sweedler_decomposition}.

If we restrict $\bbar_{\D{\O{O}}}$ to the subcategory of $\D{\O{O}}$-algebras in $\grchain$, then it is the right adjoint of an adjoint pair $\cobar_{\D{\O{O}}}\dashv\bbar_{\D{\O{O}}}$,
\begin{equation}\label{bar_cobar_adjunction}
    \coalgebrasqf{\grchain}{\D{\O{O}}}\mathop{\rightleftarrows}\limits^{\cobar_{\D{\O{O}}}}_{\bbar_{\D{\O{O}}}}\algebras{\grchain}{\D{\O{O}}},
\end{equation}
such that $\cobar_{\D{\O{O}}}B=(\D{\O{O}}\circ B,\partial)$ for a certain perturbation $\partial$, see \cite[\S11.3 and \S11.5.3]{loday_vallette_2012_algebraic_operads}. Here, $\D{\O{O}}\circ B$ is a free $\D{\O{O}}$-algebra. Since \eqref{dO_infinity_bar} is an equivalence, it suffices to define $\cobar_{\D{\O{O}}}\bbar_{\D{\O{O}}}A$ for $A$ a $\Oinf{\D{\O{O}}}$-algebra. In this case $\cobar_{\D{\O{O}}}\bbar_{\D{\O{O}}}A=(\D{\O{O}}\circ\bbar_{\D{\O{O}}}A,d\circ\id{A})$ where $d$ is the differential of the Koszul complex $\D{\O{O}}\circ_{\zeta}\K{\D{\O{O}}}$.

Given a $\Oinf{\D{\O{O}}}$-algebra $A$, we have a natural $\infty$-morphism
\begin{equation}\label{infinity_unit}
    A\leadsto\cobar_{\D{\O{O}}}\bbar_{\D{\O{O}}}(A)
\end{equation}
corresponding through the equivalence \eqref{dO_infinity_bar} to the unit of the adjunction \eqref{bar_cobar_adjunction},
\[\bbar_{\D{\O{O}}}(A)\To\bbar_{\D{\O{O}}}\cobar_{\D{\O{O}}}\bbar_{\D{\O{O}}}(A),\]
see \cite[\S11.4.3]{loday_vallette_2012_algebraic_operads}.

We endow $\chain$ with the projective model structure, $\grchain=\chain^{\Z}$ with the product model structure, and $\grchain^{\SS}$ with the projective model structure. 

\begin{proposition}\label{infinity_E1}
    If $A$ is a $\Oinf{\D{\O{O}}}$-algebra $A$ which is cofibrant in $\grchain$ then the morphism in $\grchain$ underlying \eqref{infinity_unit} is a weak equivalence.
\end{proposition}

\begin{proof}
    The argument is like in \cite[Theorem 11.4.4]{loday_vallette_2012_algebraic_operads}, but here we should invoke \cite[Proposition 11.5.3 (b)]{fresse_2009_modules_operads_functors} instead of \cite[Proposition 6.2.3]{loday_vallette_2012_algebraic_operads}. We apply \cite[Proposition 11.5.3 (b)]{fresse_2009_modules_operads_functors} to $A$, which is cofibrant in $\grchain$ by hypothesis, and to the inclusion of the weight $0$ part $\unit$ into the Koszul complex $\D{\O{O}}\circ_{\tm}\K{\D{\O{O}}}$, which is a weak equivalence in $\grchain^{\SS}$ by Theorem \ref{main_main}. The object $\unit$ is obviously cofibrant in $\grchain^{\SS}$. We have to check that the Koszul complex too. For this, we need to use left modules over the operad $\D{\O{O}}$, which are just $\D{\O{O}}$-algebras in $\grchain^{\SS}$, see \cite[\S3.2.9]{fresse_2009_modules_operads_functors}. They carry a semi-model category structure transferred along
    \begin{equation*}
        \grchain^{\SS}\mathop{\rightleftarrows}^{\D{\O{O}}\circ-}_{\text{forget}}\algebras{\grchain^{\SS}}{\D{\O{O}}},
    \end{equation*}
    see \cite[Theorem 12.3.A]{fresse_2009_modules_operads_functors}. The model structure on $\grchain^{\SS}$ is combinatorial with sets of generating (trivial) cofibrations inherited from the usual ones in $\chain$.

    Let us check that the Koszul complex is cofibrant as a left $\D{\O{O}}$-module. 
    This will conclude the proof since any cofibrant left $\D{\O{O}}$-module is cofibrant in $\grchain^{\SS}$ by \cite[Proposition 12.3.2]{fresse_2009_modules_operads_functors} because $\D{\O{O}}$ is also cofibrant therein, see Lemma \ref{goods} below. If we filter the Koszul complex $\D{\O{O}}\circ_{\tm}\K{\D{\O{O}}}$ by the weight of $\K{\D{\O{O}}}$, the differential strictly decreases the filtration level. Hence $F_n(\D{\O{O}}\circ_\tm\K{\D{\O{O}}})$ is obtained from $F_{n-1}(\D{\O{O}}\circ_{\tm}\K{\D{\O{O}}})$ by freely attaching the $\SS$-module $(\K{\D{\O{O}}})^{(n)}$, compare the proof of \cite[Proposition 2.12]{fresse_2009_operadic_cobar_constructions}. The object $(\K{\D{\O{O}}})^{(n)}$ is cofibrant in $\grchain^{\SS}$ by Lemma \ref{goods}. This suffices to prove that the left $\D{\O{O}}$-module $\D{\O{O}}\circ_{\tm}\K{\D{\O{O}}}$ is cofibrant.
\end{proof}

\begin{lemma}\label{goods}
    The $\SS$-modules $\D{\O{O}}$ and $\K{\D{\O{O}}}$ are cofibrant in $\grchain^{\SS}$.
\end{lemma}

\begin{proof}
    First of all, notice that any object in $\grchain^{\SS}$ with trivial differential which is arity-wise projective as a module over the corresponding symmetric group is cofibrant in $\grchain^{\SS}$. We will show that both $\D{\O{O}}$ and $\K{\D{\O{O}}}$ satisfy this property. The clearly have trivial differential. The second one is arity-wise $\K{\D{\O{O}}}(r)=(\K{\O{D}}\circ\K{\O{O}})(r)=\kk[\delta]\otimes\K{\O{O}}(r)$ by Theorem \ref{main_main}. As an $\SS_r$-module this is just an infinite direct sum of copies of $\K{\O{O}}(r)$, which is $\SS_r$-projective by the standing hypotheses of this section.

    The case of $\D{\O{O}}$ is slightly more difficult. We have seen in the proof of Proposition \ref{derived_presentation} that
    \[(\O{O}\circ\O{D})(r)=\O{O}(r)\otimes\O{D}^{\otimes^r}\]
    with $\SS_n$ acting on the tensor power by permutation of tensor factors and diagonally on the first tensor product. We want to prove that this $\SS_r$-module is projective using the fact that $\O{O}(r)$ is. The tensor power is $\kk$-projective. Hence, it suffices to show that, given a group $G$ and a $G$-module $M$, the tensor product $\kk[G]\otimes M$ with the diagonal action is isomorphic as a $G$-module to $M\otimes \kk[G]$ with $G$ acting just on the right tensor factor. It is easy to check that the map
    \begin{align*}
        k[G]\otimes M & \To M\otimes k[G],            \\
        g\otimes x    & \;\mapsto\; xg^{-1}\otimes g,
    \end{align*}
    provides the desired isomorphism.
\end{proof}




We are now ready to define the wise cobar construction in the context of derived homotopy algebras.

\begin{definition}\label{cobar_algebra}
    Let $A$ be a derived $\Oinf{\O{O}}$-algebra. The \emph{negative ideal} $I_{<0}\subset \cobar_{\D{\O{O}}}\bbar_{\D{\O{O}}}(A)$ is the $\D{\O{O}}$-algebra ideal generated by the negative horizontal degree part. The \emph{good cobar construction} is the quotient $\D{\O{O}}$-algebra
    \[\bar\cobar_{\D{\O{O}}}\bbar_{\D{\O{O}}}(A)=\cobar_{\D{\O{O}}}\bbar_{\D{\O{O}}}(A)/I_{<0}.\]
\end{definition}

By definition of $I_{<0}$, the $\D{\O{O}}$-algebra $\bar\cobar_{\D{\O{O}}}\bbar_{\D{\O{O}}}(A)$ is concentrated in horizontal degrees $\geq 0$, so it is actually an $\O{O}$-algebra in $\bichain$, see Proposition \ref{dO-algebra}.

\begin{proposition}
    Given a derived $\O{O}_\infty$-algebra $A$ and an $\O{O}$-algebra $B$ in $\bichain$, a $\D{\O{O}}$-algebra morphism $f\colon \cobar_{\D{\O{O}}}\bbar_{\D{\O{O}}}(A)\to B$ factors uniquely through the natural projection onto the good cobar construction
    \[\cobar_{\D{\O{O}}}\bbar_{\D{\O{O}}}(A)\twoheadrightarrow\bar\cobar_{\D{\O{O}}}\bbar_{\D{\O{O}}}(A)\to B.\]
\end{proposition}

\begin{proof}
    It suffices to notice that the generators of the $\D{\O{O}}$-ideal $I_{<0}$ must be in the kernel of $f$, since they have negative horizontal degree.
\end{proof}

\begin{remark}\label{reflection}
    By Proposition \ref{dO-algebra}, there is a full inclusion
    \[\algebras{\bichain}{\O{O}}\hookrightarrow\algebras{\grchain}{\D{\O{O}}}.\]
    Both categories are categories of algebras over a finitary monad in a locally presentable category, therefore they are locally presentable, see \cite[\S2.78]{adamek_rosicky_1994_locally_presentable_accessible}. Limits and filtered colimits in these categories are computed point-wise, as in the underlying categories, see \cite[\S4.3]{borceux_1994_handbook_categorical_algebra_A}. This implies that $\algebras{\bichain}{\O{O}}$ is closed under limits and filtered colimits in $\algebras{\grchain}{\D{\O{O}}}$, hence reflective by \cite[\S2.48]{adamek_rosicky_1994_locally_presentable_accessible}. We therefore have an adjoint pair
    \begin{equation}
        \algebras{\grchain}{\D{\O{O}}}\mathop{\rightleftarrows}\limits^R\algebras{\bichain}{\O{O}}
    \end{equation}
    where the left adjoint $R$ is the reflector and the right adjoint is the full inclusion. 

    We have the following commutative diagram of solid arrows given by the obvious forgetful functors
    \begin{center}
        \begin{tikzcd}
            \algebras{\grchain}{\D{\O{O}}} \ar[d, shift left = .5ex] \ar[r, shift left = .5ex, dashed, "R"] & \algebras{\bichain}{\O{O}} \ar[d, shift left = .5ex] \ar[l, shift left = .5ex] \\
            \grchain \ar[u, dashed, shift left = .5ex, "\D{\O{O}}\circ-"] \ar[r, dashed, shift left = .5ex, "t_{h\geq 0}\O{D}\circ-"] & \bichain \ar[l, shift left = .5ex] \ar[u, dashed, shift left = .5ex, "\O{O}\circ-"]
        \end{tikzcd}
    \end{center}
    Each of these forgetful functors has a left adjoint, in dashed arrows. The vertical left adjoints are the corresponding free algebra functors. The bottom left adjoint first sends a graded complex $X$ to the free unbounded bicomplex $\O{D}\circ X$, see Remark \ref{dual_numbers}, and then kills all the negative horizontal part by means of the naive horizontal truncation functor $t_{h\geq 0}\colon\bichainu\to\bichain$, which is the left adjoint of the full inclusion $\bichain\hookrightarrow\bichainu$. The diagram of left adjoints (dashed arrows) commutes up to natural isomorphism, by uniqueness of adjoints.
\end{remark}

\begin{corollary}\label{cofibrants}
    If $A$ is a derived $\O{O}_\infty$-algebra then $R\cobar_{\D{\O{O}}}\bbar_{\D{\O{O}}}(A)=\bar\cobar_{\D{\O{O}}}\bbar_{\D{\O{O}}}(A)$. 
\end{corollary}


\begin{corollary}\label{wise_bar_cobar_adjunction}
    We have an adjoint pair
    \[\coalgebrasqf{\grchain}{\K{\D{\O{O}}}}\mathop{\rightleftarrows}\limits^{\bar\cobar_{\D{\O{O}}}}_{\bbar_{\D{\O{O}}}}\algebras{\bichain}{\O{O}}.\]
\end{corollary}

We now check that the quotient performed to push the cobar construction into the category of bicomplexes does not change the homotopy type if we start with a $\kk$-projective derived homotopy algebra, e.g.~a minimal model.

\begin{proposition}\label{quotient}
    If $A$ is a $\kk$-projective derived $\Oinf{\O{O}}$-algebra, then the morphism in $\grchain$ underlying natural projection $\cobar_{\D{\O{O}}}\bbar_{\D{\O{O}}}(A)\twoheadrightarrow\bar\cobar_{\D{\O{O}}}\bbar_{\D{\O{O}}}(A)$ is a weak equivalence.
\end{proposition}

\begin{proof}
    It suffices to check that the underlying graded complex of $I_{<0}$ has trivial homology. First, we filter $\cobar_{\D{\O{O}}}\bbar_{\D{\O{O}}}(A)$. Its underlying bigraded module is $\O{O}\circ\O{D}\circ\K{\O{D}}\circ\K{\O{O}}\circ A$, and we filter it according to the weight of $\K{\O{O}}$ and the horizontal degree of $A$. This filtration is bounded below and exhaustive. We have to check that this is a filtration in $\grchain$, i.e.~compatibility with the differential.

    The differential of $\cobar_{\D{\O{O}}}\bbar_{\D{\O{O}}}(A)$ consists of three pieces. This first one is induced by the  differential $d_0$ of the underlying graded complex of $A$, so it preserves the filtration. The second one, given by the perturbation $\partial^A$ of $\bbar_{\D{\O{O}}}(A)$, is compatible with the filtration because, in \eqref{bar_algebras_perturbation}, $d_k$ reduces the horizontal degree since $k>0$, and the weight of $\mu^{(1)}$ is smaller than the weight of $\mu$ because $\mu^{(2)}$ has positive weight. Actually, this piece strictly decreases the filtration level. The third part is induced by the differential of the Koszul complex of $\D{\O{O}}$. In this piece, the filtration has already been considered in the proof of Theorem \ref{main_main}. Using also the computation therein we see that
    \[E^0(\cobar_{\D{\O{O}}}\bbar_{\D{\O{O}}}(A))=\O{O}\circ(\O{D}\circ_{\kappa}\K{\O{D}})\circ\K{\O{O}}\circ A,\]
    with the differential coming from the Koszul complex of $\O{D}$ and from $d_0$ of $A$.

    We now describe the bigraded module underlying $I_{<0}$. We have that
    \[\O{O}\circ\O{D}\circ\K{\O{D}}\circ\K{\O{O}}\circ A
        =
        \bigoplus_{r\geq 0}
        \O{O}(r)\otimes_{\SS_r}(\O{D}\circ\K{\O{D}}\circ\K{\O{O}}\circ A)^{\otimes^r}\]
    with the permutation action on the tensor power.
    Denote by $C_{<0}$ the negative horizontal degree part of $\O{D}\circ\K{\O{D}}\circ\K{\O{O}}\circ A=\O{D}\otimes\K{\O{D}}\otimes(\K{\O{O}}\circ A)$ and $i\colon C_{<0}\hookrightarrow \O{D}\otimes\K{\O{D}}\otimes(\K{\O{O}}\circ A)$ its (split) inclusion in the category of graded complexes. Recall that the \emph{push-out product}  of two morphisms $f\colon X\to Y$ and $g\colon U\to V$ in a monoidal category with push-outs is the morphism $f\odot g$ in the following commutative diagram
    \begin{center}
        \begin{tikzcd}
            X\otimes U \ar[r, "\id{X}\otimes g"] \ar[d, "f\otimes \id{U}"'] \ar[rd, "\text{\scriptsize push}" description, phantom] & X\otimes V \ar[d] \ar[rdd, bend left, "f\otimes \id{V}"]& \\
            Y\otimes U \ar[r] \ar[rrd, bend right, "\id{Y}\otimes g"] & s(f\odot g) \ar[rd, "f\odot g"] & \\
            && Y\otimes V
        \end{tikzcd}
    \end{center}
    Its source will be denoted by $s(f\odot g)$, as indicated in the diagram.
    Then, the bigraded module underlying $I_{<0}$ is
    \begin{equation}\label{ideal}
        \bigoplus_{r\geq 0}\O{O}(r)\otimes_{\SS_r}s(i^{\odot^r})
    \end{equation}
    and the inclusion morphism $I_{<0}\hookrightarrow \cobar_{\D{\O{O}}}\bbar_{\D{\O{O}}}(A)$ is $\bigoplus_{r\geq 0}
        \O{O}(r)\otimes_{\SS_r}i^{\odot^r}$.

    We endow the ideal $I_{<0}\subset \cobar_{\D{\O{O}}}\bbar_{\D{\O{O}}}(A)$ with the induced filtration. We consider $(\O{D}\otimes_{\kappa}\K{\O{D}})\otimes(\K{\O{O}}\circ A)$ as a graded complex and $C_{<0}$ as a graded subcomplex. In this way, \eqref{ideal} describes the graded complex $E^0(I_{<0})$ associated to the filtration of $I_{<0}$. Hence, by Lemma \ref{spectral}, in order to conclude this proof, it suffices to prove that \eqref{ideal} has trivial homology.

    The bigraded module $\K{\O{O}}\circ A$ is $\kk$-projective since $A$ is and $\K{\O{O}}$ is arity-wise projective over the corresponding symmetric group. The bigraded module $\O{D}\circ\K{\O{D}}$ is also $\kk$-projective. The homology of the Koszul complex $\O{D}\otimes_{\kappa}\K{\O{D}}$ is $\unit$. Hence, the homology of $(\O{D}\otimes_{\kappa}\K{\O{D}})\otimes\K{\O{O}}\circ A$ coincides with the homology of $\K{\O{O}}\circ A$, which is concentrated in non-negative horizontal degrees, because $A$ is. Therefore, $C_{<0}$ has trivial homology. Now, by Mayer--Vietoris and induction we see that $s(i^{\odot^r})$ has trivial homology. Since $\O{O}$ is also arity-wise projective over the corresponding symmetric group, we conclude that \eqref{ideal} has trivial homology too.

\end{proof}


\begin{definition}\label{totalization}
    The \emph{total complex} $\tot(X)$ of a bicomplex $X$ is
    \[\tot(X)_n=\bigoplus_{p+q=n}X_{p,q}\]
    with differential $d_h+d_v$. This construction defines a \emph{totalization} functor
    \[\tot\colon\bichain\To\chain.\]
\end{definition}

\begin{remark}\label{totalization_remark}
    The totalization functor strictly preserves tensor products and related constraints, so it lifts to algebras,
    \begin{equation*}
        \tot\colon\algebras{\bichain}{\O{O}}\To\algebras{\chain}{\O{O}}.
    \end{equation*}
    This functor is left inverse to the full inclusion $\algebras{\chain}{\O{O}}\hookrightarrow\algebras{\chain}{\O{O}}$.
\end{remark}

We can finally state our strictification theorem, which says that any differential graded $\O{O}$-algebra can be recovered from a minimal model up to quasi-isomorphism.

\begin{theorem}\label{recover_dg}
    Let $A$ be a differential graded $\O{O}$-algebra and $f\colon \minimal{A} \leadsto A$ a minimal model.
    Consider the image $\bbar_{\D{\O{O}}}f\colon \bbar_{\D{\O{O}}}\minimal{A} \to \bbar_{\D{\O{O}}}A$ of $f$ along the equivalence \eqref{dO_infinity_bar}, its adjoint morphism $\bar\cobar_{\D{\O{O}}}\bbar_{\D{\O{O}}}\minimal{A} \to A$ in $\algebras{\bichain}{\O{O}}$ along the adjunction in Corollary \ref{wise_bar_cobar_adjunction}, and its totalization $\tot\bar\cobar_{\D{\O{O}}}\bbar_{\D{\O{O}}}\minimal{A} \to A$, which is a morphism in $\algebras{\chain}{\O{O}}$. This last morphism is a quasi-isomorphism.
\end{theorem}

\begin{proof}
    The map $f$ decomposes as
    \begin{center}
        \begin{tikzcd}
            \minimal{A} \ar[r, rightsquigarrow] &
            \cobar_{\D{\O{O}}}\bbar_{\D{\O{O}}}\minimal{A} \ar[r] &
            \bar \cobar_{\D{\O{O}}}\bbar_{\D{\O{O}}}\minimal{A} \ar[r] &
            A.
        \end{tikzcd}
    \end{center}
    The first $\infty$-morphism is that in \eqref{infinity_unit}. It is an $E_2$-equivalence, since it actually induces an isomorphism on the $E^1$ page by Proposition \ref{infinity_E1}. This proposition applies because the graded complex underlying $A$ is $\kk$-projective with trivial differential, so it is cofibrant. The second morphism is the one in Proposition \ref{quotient}, which is an $E^2$-equivalence for the same reason. Since $f$ is an $E^2$-equivalence, the 2-out-of-3 property implies that the third arrow, which is the one in the statement, is an $E^2$-equivalence. Therefore its totalization is a quasi-isomorphism by Lemma \ref{spectral}.
\end{proof}

In Remark \ref{cofibrant_resolution} below we see that $\tot\bar\cobar_{\D{\O{O}}}\bbar_{\D{\O{O}}}\minimal{A} \to A$ in the previous theorem is a cofibrant replacement in $\algebras{\bichaintotal}{\O{O}}$ whenever this category carries the projective model structure.

\begin{remark}\label{formal_complex}
    A minimal derived homotopy $\O{O}$-algebra structure is \textit{trivial} if it reduces to its underlying $\D{\O{O}}$-algebra structure, see Remark \ref{minimal_remark_1}. If the differential graded $\O{O}$-algebra $A$ has a trivial minimal model $\minimal{A}$, then taking vertical homology on $f\colon\minimal{A}\leadsto A$ we see that $\minimal{A}$ is also a minimal model of $H_*(A)$, so $A$ is formal by Theorem \ref{recover_dg}. If $\kk$ has projective dimension $\geq 2$ there are complexes which are not formal. Obviously, a differential graded $\O{O}$-algebra with non-formal underlying complex cannot be formal. As a consequence, most differential graded $\O{O}$-algebras have necessarily non-trivial minimal models. Sagave's \cite[Example 5.1]{sagave_2010_dgalgebras_derived_algebras} fits into this case, since $\kk=\Z/(p^2)$ and the underlying complex is
    \[\cdots \to 0\to \Z/(p^2)\stackrel{p}{\To}\Z/(p^2)\to0\to\cdots.\]

    Since $\Z$ and $\Q[t]$ have projective dimension $1$, Examples \ref{example_dugger_shipley}, \ref{commutative_minimal}, and \ref{Lie_minimal} are non-trivial for deeper reasons. In the case of Example \ref{example_dugger_shipley}, we show below in Example \ref{dugger_shipley_non_formal} that it is non-formal, hence there cannot be a trivial minimal model. The same can be done for Examples \ref{commutative_minimal} and \ref{Lie_minimal}. We will give an easy proof of this claim in a forthcoming paper on universal Massey products for operadic algebras over commutative rings, which extends part of \cite{dimitrova_2012_obstruction_theory_operadic}.
\end{remark}


\begin{proposition}\label{cofibrantCE}
    If $A$ is a $\kk$-projective minimal $\Oinf{\D{\O{O}}}$-algebra then $\bar\cobar_{\D{\O{O}}}\bbar_{\D{\O{O}}}(A)$ is cofibrant in $\algebras{\bichainCE}{\O{O}}$. 
\end{proposition}

\begin{proof}
    We filter $\bbar_{\D{\O{O}}}(A)=(\K{\D{\O{O}}}\circ A,\partial^A)$ by the weight of $\K{\D{\O{O}}}$. This actually defines a filtration because, since $A$ is minimal,  the differential of $\bbar_{\D{\O{O}}}(A)$ is just $\partial^A$, which strictly decreases the filtration level, see \eqref{bar_algebras_perturbation}. Indeed, in the first summation of \eqref{bar_algebras_perturbation}, $(\delta^j;\mu)$ has less weight than $(\delta^i;\mu)$ since $k>0$, so $j<i$. In the second summation, $\mu^{(2)}$ has positive weight, so $\mu^{(1)}$ has less weight than $\mu$.

    The previous filtration of $\bbar_{\D{\O{O}}}(A)$ induces a $\D{\O{O}}$-algebra filtration of $\cobar_{\D{\O{O}}}\bbar_{\D{\O{O}}}(A)=(\D{\O{O}}\circ\bbar_{\D{\O{O}}}A,d\circ\id{A})$ with $F_n\cobar_{\D{\O{O}}}\bbar_{\D{\O{O}}}(A)=(\D{\O{O}}\circ F_n\bbar_{\D{\O{O}}}A,d\circ\id{A})$. This holds because here $d$ is the Koszul complex differential of $\D{\O{O}}\circ_\zeta\K{\D{\O{O}}}$, which strictly decreases the filtration level. Hence, $F_n\cobar_{\D{\O{O}}}\bbar_{\D{\O{O}}}(A)$ is obtained from $F_{n-1}\cobar_{\D{\O{O}}}\bbar_{\D{\O{O}}}(A)$ by freely attaching the graded complex $(\K{\D{\O{O}}})^{(n)}\circ A$ in $\algebras{\grchain}{\D{\O{O}}}$. The attaching map is defined by $\partial^A$ and $d\circ\id{A}$, compare the proof of \cite[Proposition 2.12]{fresse_2009_operadic_cobar_constructions}. If we apply the reflector $R$ in Remark \ref{reflection}, we get a filtration of $\bar\cobar_{\D{\O{O}}}\bbar_{\D{\O{O}}}(A)$ such that $F_n\bar\cobar_{\D{\O{O}}}\bbar_{\D{\O{O}}}(A)$ is obtained from $F_{n-1}\bar\cobar_{\D{\O{O}}}\bbar_{\D{\O{O}}}(A)$ by freely attaching the bicomplex $t_{h\geq 0}\O{D}\circ (\K{\D{\O{O}}})^{(n)}\circ A$ in $\algebras{\bichain}{\O{O}}$. The proposition will follow if we show that $(\K{\D{\O{O}}})^{(n)}\circ A$ is cofibrant in $\bichainCE$. Since $A$ is $\kk$-projective and each $\K{\O{O}}(r)$ is projective as a $\SS_r$-module, \[(\K{\D{\O{O}}})^{(n)}\circ A=\bigoplus_{\substack{r\geq 0\\u+v=n}}\kk\cdot \delta^u\otimes(\K{\O{O}})^{(v)}(r)\otimes_{\SS_r}A^{\otimes^r}\] is a $\kk$-projective graded complex. It has trivial differential because $A$ is minimal. Therefore, $t_{h\geq 0}\O{D}\circ (\K{\D{\O{O}}})^{(n)}\circ A$ is a retract of a direct sum of copies of the bicomplexes $S^{0,q}$, $\partial_vD^{p,q}$, $p>0$, $q\in\Z$, which are cofibrant in $\bichainCE$, see \cite[\S4]{muro_roitzheim_2019_homotopy_theory_bicomplexes}.
\end{proof}

The following theorem is the improved version of Proposition \ref{minimal_exist}. It shows that minimal models are essentially unique.

\begin{theorem}\label{compare_minimal_models}
    Any $\O{O}$-algebra $A$ has a minimal model $\minimal{A}'\leadsto A$ such that, if $\minimal{A}\leadsto A$ is another minimal model, then there exists an $E^2$-equivalence $\minimal{A}\leadsto \minimal{A}'$ which actually induces the identity in $H_*(A)$ on $E^2$ terms.
\end{theorem}

\begin{proof}
    The minimal model $i\colon \minimal{A}'\leadsto A$ is that constructed in the proof of Proposition \ref{minimal_exist} from a cofibrant resolution $\cofres{A}\to A$ in $\algebras{\bichain}{\O{O}}$. In particular the graded complex (with trivial differential) underlying $\minimal{A}'$ is the vertical homology $H_*(\cofres{A})$, and $\minimal{A}'\leadsto A$ factors through the cofibrant resolution
    \[\minimal{A}'\leadsto \cofres{A}\to A.\]
    Under our hypotheses, Berglund's version of the homotopy transfer theorem \cite[Theorem 1.3]{berglund_2014_homological_perturbation_theory} yields a left inverse $\cofres{A}\leadsto \minimal{A}'$ to $i$. This left inverse is an $E^2$-equivalence by the 2-out-of-3 property.

    Let $\minimal{A}\leadsto A$ be another minimal model. We consider the following diagram, where the bottom row is the decomposition of $\minimal{A}\leadsto A$ in the proof of Theorem \ref{recover_dg}
    \begin{center}
        \begin{tikzcd}
            &&&\cofres{A}\ar[d, ->>]\\
            \minimal{A} \ar[r, rightsquigarrow] &
            \cobar_{\D{\O{O}}}\bbar_{\D{\O{O}}}\minimal{A} \ar[r] &
            \bar \cobar_{\D{\O{O}}}\bbar_{\D{\O{O}}}\minimal{A} \ar[r] \ar[ru, dashed]&
            A.
        \end{tikzcd}
    \end{center}
    The vertical arrow is a trivial fibration because it is a cofibrant resolution in $\algebras{\bichain}{\O{O}}$. Moreover, $\bar \cobar_{\D{\O{O}}}\bbar_{\D{\O{O}}}\minimal{A}$ is cofibrant in this semi-model category by Proposition \ref{cofibrantCE}. Hence, there is a lifting (dashed arrow). All arrows here are $E^2$-equivalences by the 2-out-of-3 property, see the proof of Proposition \ref{minimal_exist}. The desired $E^2$-equivalence $\minimal{A}\leadsto \minimal{A}'$ is the composition of $\minimal{A}\leadsto \cofres{A}$ in the previous diagram with the derived $\infty$-morphism $\cofres{A}\leadsto \minimal{A}'$ above.
\end{proof}



The same argument proves the following result.

\begin{proposition}\label{commutative_square}
    Any differential graded $\O{O}$-algebra $A$ has a minimal model $\minimal{A}'\leadsto A$ such that, if $f\colon B\to A$ is an $\O{O}$-algebra morphism and $\minimal{A}\leadsto B$ is a minimal model, then there exists an morphism of derived homotopy $\O{O}$-algebras $\minimal{A}\leadsto \minimal{A}'$ which actually induces $H_*(f)$ on $E^2$ terms.
\end{proposition}

\section{The total model structure}\label{total_section}

So far we have only considered the Cartan--Eilenberg model structure on the category $\bichain$ of bicomplexes. In this section we consider the \emph{total model structure} $\bichaintotal$, which is Quillen equivalent to $\chain$ (with the projective model structure), see \cite[\S3]{muro_roitzheim_2019_homotopy_theory_bicomplexes}. We extend this to algebras over an operad in $\chain$.

Weak equivalences in $\bichaintotal$ are the morphisms whose totalization (see Definition \ref{totalization}) is a quasi-isomorphism in $\chain$. A morphism of bicomplexes is a (trivial) fibration in $\bichaintotal$ if it is bidegree-wise surjective and it induces an isomorphism in vertical homology in positive (resp.~all) degrees.

Consider the adjoint pair
\begin{equation}\label{transfer_classic}
    \chain\mathop{\rightleftarrows}\limits^{\O{O}\circ-}_U\algebras{\chain}{\O{O}}.
\end{equation}
analogue to \eqref{underlying_bicomplex} for complexes instead of bicomplexes.

\begin{proposition}\label{model_derived_algebras_total}
    Let $\O{O}$ be an operad in $\chain$ such that $\algebras{\chain}{\O{O}}$ admits the transferred model structure along \eqref{transfer_classic}. Then the category $\algebras{\bichain}{\O{O}}$ of $\O{O}$-bialgebras admits the transferred model structure from the total model structure $\bichaintotal$ along \eqref{underlying_bicomplex}, and we denote it by $\algebras{\bichaintotal}{\O{O}}$.
\end{proposition}

\begin{proof}
    We are going to check that the conditions in \cite[Theorem 11.3.2]{hirschhorn_2003_model_categories_their} are satisfied. Since $\algebras{\bichain}{\O{O}}$ is a locally presentable category, it suffices to prove that, if $J_{\tot}$ denotes the set of generating trivial cofibrations for $\bichaintotal$ \cite[\S3]{muro_roitzheim_2019_homotopy_theory_bicomplexes}, then any relative $\O{O}\circ J_{\tot}$-cell complex $f\colon A\to B$ is a weak equivalence in $\bichain$, i.e.~$\tot(f)$ is a quasi-isomorphism in $\chain$.

    The totalization functor in Definition \ref{totalization} preserves all (co)limits, which are computed point-wise in source and target. It also preserves tensor products, see Remark \ref{totalization}. Therefore, the usual construction of the push-out of a free map in a category of algebras over an operad (see e.g.~\cite[\S7]{harper_2010_homotopy_theory_modules}) shows that the totalization of a relative $\O{O}\circ J_{\tot}$-cell complex $f$ is a relative relative $\O{O}\circ \tot(J_{\tot})$-cell complex $\tot(f)$ in $\chain$. Maps in $\tot(J_{\tot})$ are trivial cofibrations in $\chain$, see \cite[\S3]{muro_roitzheim_2019_homotopy_theory_bicomplexes}, hence $\tot(f)$ is a quasi-isomorphism by the hypothesis on $\algebras{\chain}{\O{O}}$.
\end{proof}

\begin{proposition}\label{quillen_equivalence_derived_algebras_full_inclusion}
    Let $\O{O}$ be an operad in $\chain$ such that $\algebras{\chain}{\O{O}}$ admits the transferred model structure along \eqref{transfer_classic}.
    The inclusion $i\colon \algebras{\chain}{\O{O}}\hookrightarrow\algebras{\bichaintotal}{\O{O}}$ in Remark \ref{vertical_inclusion} is a left Quillen equivalence.
\end{proposition}

\begin{proof}
    The functor $i$ has a right adjoint given by $A\mapsto A_{0,*}$. Indeed, the horizontal degree $0$ part of an $\O{O}$-algebra in in $\bichain$ is naturally an $\O{O}$-algebra in $\chain$. This right adjoint clearly preserves (trivial) fibrations, i.e.~it is a right Quillen functor, so $i$ is a left Quillen functor. The proof of \cite[Proposition 3.4]{muro_roitzheim_2019_homotopy_theory_bicomplexes} shows that, for $B$ an arbitrary $\O{O}$-algebra in $\chain$ and $A$ an $\O{O}$-algebra in $\bichain$, a morphism $f\colon i(B)\to A$ in $\algebras{\bichaintotal}{\O{O}}$ is a weak equivalence if and only if the adjoint morphism $f_{0,*}\colon B\to A_{0,*}$ is a weak equivalence in $\algebras{\chain}{\O{O}}$. This proves that the previous Quillen pair is a Quillen equivalence.
\end{proof}

The following result is a version of Proposition \ref{cofibrantCE} for the total model structure.

\begin{proposition}\label{cofibrant}
    Under the standing assumptions of \S\ref{cobar_section},
    if $A$ is a $\kk$-projective minimal $\Oinf{\D{\O{O}}}$-algebra then $\bar\cobar_{\D{\O{O}}}\bbar_{\D{\O{O}}}(A)$ is cofibrant in $\algebras{\bichaintotal}{\O{O}}$.
\end{proposition}

\begin{proof}
    Under the assumptions of \S\ref{cobar_section}, $\algebras{\chain}{\O{O}}$ admits the transferred model structure along \eqref{transfer_classic}, see \cite{hinich_1997_homological_algebra_homotopy,hinich_2003_erratum_homological_algebra,muro_2011_homotopy_theory_nonsymmetric,muro_2017_correction_articles_homotopy}.
    The proof is the same as that of Proposition \ref{cofibrantCE}. One just has to notice that the bicomplexes $S^{0,q}$, $\partial_vD^{p,q}$, $p>0$, $q\in\Z$, are also cofibrant in $\bichainCE$ by \cite[\S3]{muro_roitzheim_2019_homotopy_theory_bicomplexes}.
\end{proof}

\begin{remark}\label{cofibrant_resolution}
    The argument in the proof of Proposition \ref{model_derived_algebras_total} also shows that the totalization of a cofibrant object in $\algebras{\bichaintotal}{\O{O}}$ is cofibrant in $\algebras{\chain}{\O{O}}$ since the totalization of a generating cofibration in $\bichaintotal$ is a cofibration in $\chain$, compare \cite[\S3]{muro_roitzheim_2019_homotopy_theory_bicomplexes}. Therefore, under the assumptions of \S\ref{cobar_section}, Theorem \ref{recover_dg} produces a cofibrant replacement of a differential graded $\O{O}$-algebra out of any minimal model.
\end{remark}

\begin{corollary}\label{representatives}
    Under the standing assumptions of \S\ref{cobar_section},
    given two differential graded $\O{O}$-algebras $A,B$ and a minimal model $\minimal{A}\leadsto A$, maps $A\to B$ in the homotopy category of $\algebras{\O{O}}{\chain}$ are represented by derived $\infty$-morphisms $\minimal{A}\leadsto B$. A derived $\infty$-morphism $\minimal{A}\leadsto B$ represents an isomorphism in the homotopy category if and only if it is an $E^2$-equivalence.
\end{corollary}

\begin{proof}
    By Proposition \ref{quillen_equivalence_derived_algebras_full_inclusion}, we can place $A, B$ in $\algebras{\bichaintotal}{\O{O}}$. They are fibrant here since they are concentrated in horizontal degree $0$. By Theorem \ref{recover_dg} and Proposition \ref{cofibrant}, $\bar\cobar_{\D{\O{O}}}\bbar_{\D{\O{O}}}(\minimal{A})$ is a cofibrant replacement of $A$, so maps $A\to B$ in the homotopy category of $\algebras{\O{O}}{\chain}$ are represented by honest maps $\bar\cobar_{\D{\O{O}}}\bbar_{\D{\O{O}}}(\minimal{A})\to B$ in $\algebras{\bichaintotal}{\O{O}}$. By Corollary \ref{wise_bar_cobar_adjunction} and \eqref{dO_infinity_bar}, a map $\bar\cobar_{\D{\O{O}}}\bbar_{\D{\O{O}}}(\minimal{A})\to B$ is essentially the same thing as a derived $\infty$-morphism $\minimal{A}\leadsto B$.

    The $E^2$-term of the spectral sequence of $A_M$ is $H_*(A)$ concentrated in horizontal degree $0$ by definition of minimal model. By Propositions \ref{infinity_E1} and \ref{quotient} and standard spectral sequence arguments, a map $\bar\cobar_{\D{\O{O}}}\bbar_{\D{\O{O}}}(\minimal{A})\to B$ is a weak equivalence in $\algebras{\bichaintotal}{\O{O}}$ if and only if the corresponding derived $\infty$-morphism $\minimal{A}\leadsto B$ is an $E^2$-equivalence.
\end{proof}

\begin{remark}
    It should be possible to work out a notion of derived $\infty$-homotopy, like in \cite{berglund_2014_homological_perturbation_theory,dotsenko_poncin_2016_tale_three_homotopies}, to describe the equivalence relation on the set of derived $\infty$-morphisms $\minimal{A}\leadsto B$ whose quotient is the set of maps $A\to B$ in the homotopy category. For $\O{O}=\O{A}$ the associative operad, derived $\infty$-homotopies should coincide with filtration-preserving $\infty$-homotopies in the sense of \cite[Definition 1.2.1.7]{lefevre-hasegawa_2003_categories} and \S\ref{presentation}. These derived $\infty$-homotopies were considered in \cite{cirici_egas_santander_livernet_whitehouse_2018_derived_infinity_algebras} under the name of $0$-homotopies. They also defined $0$-homotopies for twisted complexes, which should be the right notion of derived $\infty$-homotopy for $\O{O}=\unit$ the initial operad.
\end{remark}

\begin{example}\label{dugger_shipley_non_formal}
    We here apply Corollary \ref{representatives} to show by contradiction that the differential graded associative algebra of Dugger and Shipley considered in Example \ref{example_dugger_shipley} is not formal. Otherwise, there should be an $E^2$-equivalence of derived homotopy associative algebras $f\colon \minimal{A}\leadsto H_*(A)$ given by bidegree $(-i,r-1+i)$ maps $f_{i,r}\colon \minimal{A}^{\otimes^r}\to H_*(A)$, $i\geq 0$, $r\geq 1$. We can assume without loss of generality that $f$ induces the identity on $E^2$ terms, therefore $f_{0,1}(x^n)=x^n$, $n\in\Z$, and $f_{0,1}(cx^n)=0$ for degree reasons.

    The graded algebra $H_*(A)$ regarded as a derived homotopy associative algebra satisfies $m_{i,r}=0$ and $d_j=0$ for $(i,r)\neq (0,2)$ and all $j$. The minimal model $\minimal{A}$ in Example \ref{example_dugger_shipley} satisfies the same equations for $(i,r)\neq (0,2),(1,2)$ and $j\neq 1$. The operation $m_{0,2}$ is the binary associative product in both cases. For degree reasons, $f_{i,r}=0$ for $i>r$. Moreover, $d_1(\minimal{A})\subset p\cdot \minimal{A}$ and $p\cdot H_*(A)=0$. Therefore, the derived $\infty$-morphism equation in Example \eqref{examples_derived_infty-morphisms} (1) reduces to
    \begin{multline*}
        0=\sum_{l=1}^{r-1}(-1)^{l+1}\left(f_{i,r-1}(x_{1},\dots,x_{l}x_{l+1},\dots)+f_{i-1,r-1}(x_{1},\dots,m_{1,2}(x_{l},x_{l+1}),\dots)\right)\\
        -\sum_{k=0}^i\sum_{t=1}^{r-1}(-1)^{(t+1)+(r-t-1)\sum\limits_{u=1}^{t}\abs{x_u}}f_{k,t}(x_{1},\dots)f_{i-k,r-t}(\dots, x_{r}),
    \end{multline*}
    for all $i\geq 0$ and $r\geq 1$.

    For $(i,r)=(1,2)$ we obtain
    \begin{align*}
        f_{1,1}(x_1x_2)+f_{0,1}m_{1,2}(x_1,x_2) & =f_{1,1}(x_1)f_{0,1}(x_2)+f_{0,1}(x_1)f_{1,1}(x_2).
    \end{align*}
    Taking $(x_1,x_2)=(c,x)$,
    \begin{align*}
        f_{1,1}(cx) & =f_{1,1}(c)x.
    \end{align*}
    Taking $(x_1,x_2)=(x,c)$,
    \begin{align*}
        f_{1,1}(xc)+f_{0,1}m_{1,2}(x,c) & = xf_{1,1}(c).
    \end{align*}
    By the relations in $\minimal{A}$,
    \[xc=-cx,\]
    see Example \ref{example_dugger_shipley}.
    For degree reasons, $f_{1,1}(c)=a\cdot x$ for some $a\in\F_p$, hence
    \[xf_{1,1}(c)=f_{1,1}(c)x.\]
    Using the formula for $m_{1,2}$ in Example \ref{example_dugger_shipley} we obtain
    \begin{align*}
        -f_{1,1}(c)x+x^{2} & = f_{1,1}(c)x,
    \end{align*}
    i.e.
    \[2f_{1,1}(c)x=x^{2}.\]
    This yields a contradiction if $p=2$. Otherwise, it yields the formula
    \[f_{1,1}(c)=\frac{1}{2}x.\]

    Suppose $p\neq 2$. For $(i,r)=(2,2)$, the derived $\infty$-morphism equation is
    \begin{align*}
        f_{1,1}m_{1,2}(x_1,x_2) & = f_{1,1}(x_1)f_{1,1}(x_2).
    \end{align*}
    In particular,
    \begin{align*}
        f_{1,1}m_{1,2}(c,c) & = f_{1,1}(c)f_{1,1}(c)=\frac{1}{4}x^{2},
    \end{align*}
    but the left hand side is $0$ by the formula for $m_{1,2}$ in Example \ref{example_dugger_shipley}. This is a contradiction.
\end{example}

The interested reader can prove the non-formality of Examples \ref{commutative_minimal} and \ref{Lie_minimal} along the same lines. We will give an elementary proof of this fact in a forthcoming paper by using secondary operations.
















\newcommand{\etalchar}[1]{$^{#1}$}
\providecommand{\bysame}{\leavevmode\hbox to3em{\hrulefill}\thinspace}
\providecommand{\MR}{\relax\ifhmode\unskip\space\fi MR }
\providecommand{\MRhref}[2]{%
    \href{http://www.ams.org/mathscinet-getitem?mr=#1}{#2}
}
\providecommand{\href}[2]{#2}

\end{document}